\documentclass[11pt,a4paper]{amsart}
\usepackage{amssymb,amsmath}
\usepackage{amsfonts}
\usepackage{verbatim}
\usepackage{enumerate}
\usepackage{mathabx}
\usepackage{hyperref}

\def\NN{\mathbb{N}}
\def\RR{\mathbb{R}}

\def\O{\mathcal{O}}

\theoremstyle{plain}
\newtheorem{thm}{Theorem}
\newtheorem{lem}[thm]{Lemma}
\newtheorem{prop}[thm]{Proposition}

\theoremstyle{remark}
\newtheorem{rema}{Remark}
\theoremstyle{definition}
\newtheorem{defin}{Definition}

\title[Asymptotic variance for RWM chains in high dimensions]{
Asymptotic variance for Random walk Metropolis chains in high dimensions: 
logarithmic growth via the Poisson equation}

\author{Aleksandar Mijatovi\'{c}}
\address{Department of Statistics, University of Warwick, and The Alan Turing Institute, UK}
\email{aleks.mijatovic@gmail.com}

\author{Jure Vogrinc}
\address{Department of Statistics, University of Warwick, UK}
\email{Jure.Vogrinc@warwick.ac,uk}

\thanks{We thank Krys Latuszynski for insightful comments.
AM is supported in part by the EPSRC grant EP/P003818/1 and the Programme on Data-Centric Engineering  funded by
Lloyd's Register Foundation. JV is supported by President's PhD Scholarship of Imperial College London and by the EPSRC grants  EP/P002625/1 and EP/R022100/1}



\begin{document}

\begin{abstract}
There are two ways of speeding up MCMC algorithms: (1) construct more
complex samplers that use gradient and higher order information about
the target and (2) design a control variate to reduce the asymptotic
variance. While the efficiency of (1) 
as a function of dimension has been studied extensively, this paper provides 
the first results linking the efficiency of (2) with dimension.
Specifically, we construct a control variate for a  
$d$-dimensional Random walk Metropolis chain with an IID target 
using the solution of the Poisson equation for the 
scaling limit in~\cite{RobertsGelmanGilks97}. 
We prove that 
the asymptotic variance of the corresponding estimator is bounded above by a multiple of 
$\log(d)/d$ over the spectral gap of the chain. 
The proof hinges on 
large deviations theory, optimal Young's inequality and
Berry-Esseen type bounds. 
Extensions of the result to non-product targets are discussed. 
\end{abstract}

\date{\today}

\maketitle

\section{Introduction}

Markov Chain Monte Carlo (MCMC) methods are designed to approximate 
expectations of high dimensional random vectors, see e.g.~\cite{Handbook, tierney}. 
It is hence important to understand how the efficiency of MCMC algorithms
scales with dimension. The optimal scaling literature, initiated by the seminal
paper~\cite{RobertsGelmanGilks97}, indicates that for the high-dimensional algorithms
it is the growth of the asymptotic variance with dimension that 
provides perhaps the most natural measure of efficiency for MCMC (see~\cite[Sections~1.2 and~2.2]{RobertsRosenthal01},
\cite[Sec.~3]{MALAScaling}).
For instance, for a product target, the asymptotic variance for the Random walk Metropolis (RWM) chain on
$\RR^d$ is 
heuristically of the order $\O(d)$  \cite{RobertsGelmanGilks97,RobertsRosenthal01}. 
Moreover, 
the asymptotic variances 
of the $d$-dimensional 
Metropolis-adjusted Langevin algorithm (MALA), 
Hamiltonian Monte Carlo  and the fast-MALA
are
$\O(d^{1/3})$ \cite{MALAScaling}, $\O(d^{1/4})$ \cite{HMCScaling} and $\O(d^{1/5})$ \cite{fMALA}, respectively.

This paper constructs a dimension-dependent estimator (see~\eqref{eq:CLT} below)
and proves a bound on its asymptotic variance, 
suggesting the order $\O(\log(d))$, for a RWM chain with an IID target.
The idea is to exploit the following facts:
(I) the law of the diffusion scaling limit for the RWM chain 
(as $d\to\infty$) from~\cite{RobertsGelmanGilks97}
is close (in the weak sense) to that of the law of the chain 
itself \textit{and} (II) the Poisson equation for the limiting Langevin diffusion
has an explicit solution. 
Following ideas from~\cite{MV},  
we construct and analyse the estimator using (I) and (II).

Specifically, let $\rho$ be a density on $\RR$ and 
$\rho_d(\mathbf{x}^d):=\prod_{i=1}^d\rho(x^d_i)$
the corresponding $d$-dimensional product density, 
where
$\mathbf{x}^d=(x_1^d,\dots,x_d^d)\in\RR^d$.
Let
$\mathbf{X}^d=\{\mathbf{X}^d_n\}_{n\in\NN}$
be the RWM chain converging to $\rho_d$
with the normal proposal with variance 
$l^2/d\cdot I_d$
(here 
$I_d\mathbf{x}^d=\mathbf{x}^d$, for all $\mathbf{x}^d\in\RR^d$, and 
$l$ a constant),  
analysed in~\cite{RobertsGelmanGilks97}.
If
$f(\mathbf{x}^d)$ depends only on the first coordinate $x_1^d$, 
then 
$\rho_d(f):=\int_{\RR^d} f(\mathbf{x}^d)\rho_d(\mathbf{x}^d)d\mathbf{x}^d= \int_\RR f(x)\rho(x)dx=:\rho(f)$.
Under appropriate conditions, 
the asymptotic variance $\sigma^2_{f,d}$ in the Central limit theorem (CLT) for the 
estimator
$\sum_{i=1}^n f(\mathbf{X}^d_i)/n$
of
$\rho_d(f)$
satisfies
$$\sigma_{f,d}^2\leq 2\mathrm{Var}_\rho(f)/(1-\lambda_d)\qquad \text{and, heuristically, 
}\qquad\sigma_{f,d}^2= \O(d)\text{ as $d\to\infty$.}$$
Here 
$\mathrm{Var}_\rho(f):=\rho(f^2)-\rho(f)^2$ and 
$1-\lambda_d$ denotes the spectral gap of the chain 
$\mathbf{X}^d$. The inequality follows by the spectral representation of $\sigma_{f,d}^2$ 
in~\cite{Geyer92,KipnisVaradhan86}.
The reasoning analogous to that 
applied to the integrated autocorrelation time
in \cite[Sec.~2.2]{RobertsRosenthal01} can be used to argue that 
the spectral gap $1-\lambda_d$ is of the order $\O(1/d)$.
Hence the asymptotic variance $\sigma_{f,d}^2$ is $\O(d)$.

The Poisson equation for the Langevin diffusion arising in the scaling limit of  $\mathbf{X}^d$
in~\cite{RobertsGelmanGilks97} is a second order linear ODE with solution $\hat f$ given explicitly in 
terms of $f$ and the density $\rho$, see~\eqref{eq:solutions_Poisson_Eq_diff} below.
For large $d$, the function $\hat f$ ought to approximate the solution of the Poisson equation for
$\mathbf{X}^d$. 
The reasoning in~\cite{MV} then suggests the form of an estimator for $\rho_d(f)$, which 
under appropriate technical assumptions,
satisfies the following CLT:
\begin{equation}\label{eq:CLT}
\sqrt{n}\left(\frac{1}{n}\sum_{i=1}^n\left(f+d(P_d\hat{f}-\hat{f})\right)(\mathbf{X}^d_i)-\rho_d(f)\right)\xrightarrow{d}
N(0,\hat \sigma^2_{f,d})\quad\text{as }n\to\infty\text{,} 
\end{equation}
where $P_d$ is the transition kernel of the chain $\mathbf{X}^d$. 
The main result of this paper (Theorem~\ref{thm:AVbound} below) states that for some constant $C>0$ 
the following inequality 
holds 
$$\hat \sigma_{f,d}^2\leq  C\log(d)/(d(1-\lambda_d))  \qquad \text{and, heuristically,}\qquad\hat \sigma_{f,d}^2=\O(\log(d))\text{ as $d\to\infty$.}$$
Theorem~\ref{thm:AVbound} also gives the explicit dependence of the constant $C$ on the function $f$. 

This result suggests that to achieve the same level
of accuracy as when estimating $\rho_d(f)$ by an average over an IID sample 
form $\rho_d$,
only  $\O(\log d)$ times as many RWM samples are
needed if the control variate $d (P_d\hat{f}-\hat{f})$ is added.  This should be
contrasted to  $\O(d)$ (resp. $\O(d^{1/3})$, $\O(d^{1/4})$, $\O(d^{1/5})$) times as many samples
for the RWM (resp. MALA, Hamiltonian Monte Carlo, fast-MALA) 
without the control
variate, 
see~\cite{RobertsRosenthal01, MALAScaling, HMCScaling, fMALA}. 

The optimal scaling for the proposal variance of a $d$-dimensional RWM chain is $\O(1/d)$, 
see~\cite{RobertsRosenthal01}
for a review  
and~\cite[Thm.~4]{Optimal_Scaling_Var} for the proof that other scalings lead to suboptimal behaviour. 
To get a non-trivial scaling limit in~\cite{RobertsGelmanGilks97},
it is necessary to accelerate the chains $(\mathbf{X}^d)_{d\in\NN}$ linearly 
in dimension. The weak convergence of the accelerated chain to the Langeivn diffusion
suggests that 
$d\hat{f}$ is close to the solution of the Poisson equation for $P_d$ and $f$,
making $d (P_d\hat{f}-\hat{f})$ a good control variate. Using an
approximate solution to the Poisson equation to construct control variates is a
common variance reduction technique, see e.g.~\cite{henderson,Meyn_Control,petros, MV}. 
In this context, more often than not, 
an approximate solution used is a solution of a
Poisson equation of a simpler related process. For instance, in~\cite{MV} 
a sequence of Markov chains on a finite state space, converging weakly to a given RWM chain,
is constructed. 
Then a solutions to the  Poisson equation of the finite state chain
is used to construct a control variate capable of reducing the 
asymptotic variance of the RWM chain arbitrarily. 
Here this idea is turned on its head: 
a solution to the Poisson equation of the limiting diffusion is used to construct a
control variate for a RWM chain from a weakly convergent sequence in~\cite{RobertsGelmanGilks97}.
Since the complexity of the RWM increases arbitrarily as dimension $d\to\infty$, 
it is infeasible to get an arbitrary variance reduction as in~\cite{MV}.  
However, heuristically, the amount of 
variance reduction  measured by the ratio $\sigma_{f,d}^2/\hat \sigma_{f,d}^2$ 
still tends to infinity at the rate $d/\log(d)$.

If the solution of the Poisson equation for $\mathbf{X}^d$ and $f$ were available,
we could construct an estimator for $\rho_d(f)$ with zero variance (see e.g.~\cite{MV}).
Put differently, in this case there would be no need for the chain to explore its state space
at all.
In our setting,
since the jumps of 
$\mathbf{X}^d$
are of size 
$\O(1/\sqrt{d})$~\cite{RobertsGelmanGilks97},
after $\O(\log d)$ steps 
the chain will have explored the distance of 
$\O(\log(d)/\sqrt{d})$. In line with the observation above this distance
tends to zero as $d\to\infty$ since, heuristically, 
$d\hat f$ approximates the solution of the Poisson equation for $P_d$ and $f$
arbitrarily well. 

The key technical step 
in the proof of our result
(Theorem~\ref{thm:GdG} below)
is a type of concentration inequality.
It generalises the limit in~\cite[Lemma~2.6]{RobertsGelmanGilks97}, 
which essentially states that generators of the accelerated chains $(\mathbf{X}^d)_{d\in\NN}$ converge 
to the generator of the Langevin limit when applied to a compactly supported and infinitely smooth function,
in two ways: (A) it extends the limit to a class of functions of sub-exponential growth
and (B) provides estimates for the rate of convergence.
Both of these extensions are key for our main result. (A) allows us to apply 
Theorem~\ref{thm:GdG} to a solution of the Poisson equation, which is not compactly supported.
Note that this step in the proof entails identifying the correct space of functions that is closed under the operation of
solving the Poisson equation (see Proposition~\ref{prop:fhatinS} below). Estimate (B) allows us to control the asymptotic variance 
via a classical spectral-gap bound.
The proof of Theorem~\ref{thm:GdG}, outlined in Sec.~\ref{subsec:outline} below,
crucially depends on the large deviations theory (Sec.~\ref{sec:EventsBounds}), the form of the constant in optimal 
Young's inequality (Sec.~\ref{sec:PDFBOunds})
and Berry-Esseen type bounds (Sec.~\ref{sec:CDFandCFBounds}). 

We conclude the introduction with a comment on how the present paper fits into the literature.
Since, as discussed above, the asymptotic variance 
$\sigma_{f,d}^2$ 
is approximately equal to the product
$2\mathrm{Var}_\rho(f)/(1-\lambda_d)$, 
two ``orthogonal'' approaches to speeding up MCMC algorithms are feasible. 
(a) The MCMC method itself can be modified, with the aim of increasing the
spectral gap, leading to many well-known reversible samplers such as 
MALA and  Hamiltonian Monte Carlo~\cite{HMC}
as well as 
non-reversible ones~\cite{ZigZag,Pavliotis}. 
There is a plethora of papers (see~\cite{MALAScaling,HMCScaling,fMALA} and the references therein)
studying the asymptotic properties of such sampling algorithms as dimension increases to infinity. 
(b) A control variate $g$, satisfying $\rho(g)=0$, may be added to $f$ 
with the aim of reducing $\mathrm{Var}_\rho(f)$ to
$\mathrm{Var}_\rho(f+g)$ 
without modifying the MCMC algorithm. 
A number of control variates have been proposed in the MCMC literature \cite{AssarafCaffarel99, MiraGirolami2014,
OatesGirolami2017, petros}.  Thematically, the present paper fits under (b) and, to the best of our knowledge,
is the first to investigate the growth of the asymptotic variance as the dimension $d\to\infty$ in this context.  
Moreover, it is feasible that our method could be generalised to some of the algorithms  under (a), 
see Section~\ref{sec:Implementation} below for a discussion of possible extensions.

The remainder of paper is structured as follows. Section~\ref{Sec:A&R} 
gives a detailed description of the assumptions and states the results. 
Section~\ref{sec:Extensions} illustrates algorithms based on our main result with numerical examples and discusses (without proof) potential extensions of our results for other MCMC methods, more general targets, etc.
In Section~\ref{sec:proofs} we prove our results. Section~\ref{sec:Technical} 
develops the tools needed for the proofs of Section~\ref{sec:proofs}.
Section~\ref{sec:Technical} 
uses results from probability and analysis but
is independent of all that
precedes it in the paper. 

\section{Results}\label{Sec:A&R}

Let
$\mathbf{X}^d=\{\mathbf{X}^d_n\}_{n\in\NN}$
be a RWM chain in $\RR^d$ with a transition kernel $P_df:=(1/d)\mathcal{G}_df + f$, where 
\begin{equation}\label{eq:MetropolisGenerator}
\mathcal{G}_df(\mathbf{x}^d):=d\mathbb{E}_{\mathbf{Y}^d}\left[\left(f(\mathbf{Y}^d)-f(\mathbf{x}^d)\right) \alpha(\mathbf{x}^d,\mathbf{Y}^d)
\right],\qquad 
\alpha(\mathbf{x}^d,\mathbf{Y}^d):=1\wedge\frac{\rho_d(\mathbf{Y}^d)}{\rho_d(\mathbf{x}^d)},
\end{equation}
$\mathbf{x}^d\in\RR^d$,
$\mathbf{Y}^d=(Y_1^d,\ldots,Y_d^d)\sim N(\mathbf{x}^d,l^2/d\cdot I_d)$
and $x\wedge y:=\min\{x,y\}$ for all $x,y\in\RR$,
started in stationarity $\mathbf{X}^d_1\sim \rho_d$.
Let $\mathcal{S}^n$ consists of all the functions  with their
first $n$ derivatives growing slower then any exponential function. 
More precisely, 
for any $n\in\NN\cup\{0\}$, define
\begin{equation}\label{eq:defSn}
\mathcal{S}^n:=\left\{g\in\mathcal{C}^{n}(\RR)\colon  \sum_{i=0}^n\|g^{(i)}\|_{\infty,s}<\infty\>\>\forall s>0\right\},
\qquad
\text{where $\|g\|_{\infty,s}:=\sup_{x\in\RR}\left(e^{-s|x|}|g(x)|\right)$}
\end{equation}
and
$\mathcal{C}^n(\RR)$ (resp. $\mathcal{C}^0(\RR)$)
denotes $n$-times continuously differentiable (resp. continuous) functions.
Our main result (Theorem~\ref{thm:AVbound} below) applies to the space $\mathcal{S}^1$, 
containing functions $f$ for which 
$\rho(f):=\int_{\RR}f(x)\rho(x) dx$ is typically of interest in applications (e.g. polynomials).  
In addition, spaces in~\eqref{eq:defSn} are closed for solving Poisson's equation in~\eqref{eq:UfPE}, see Proposition~\ref{prop:fhatinS} below. 

Throughout the paper 
$\rho$ denotes a strictly positive density on $\RR$ with $\log(\rho)\in\mathcal{S}^4$ and 
\begin{equation}\label{eq:SupperExponentialTails}
\lim_{|x|\to\infty}\frac{x}{|x|}\cdot \log(\rho(x))'=-\infty\text{,}
\end{equation} 
unless otherwise stated.
Assumption~\eqref{eq:SupperExponentialTails} implies that the tails of $\rho$ decay faster
then any exponential, i.e. 
$\mathbb{E}\left[e^{sX}\right]<\infty$ 
for any $s\in\RR$ for $X\sim\rho$ (cf.~\cite[Sec.~4]{jarner}). 
The assumption $\log(\rho)\in\mathcal{S}^4$ prohibits $\rho$ from
decaying to quickly, e.g. proportionally to $e^{-e^{|x|}}$. 
Both of these assumptions serve brevity and clarity of the proofs and it is feasible they can
be relaxed. Nevertheless, 
a large class of densities of interest satisfy these
assumptions, e.g. mixtures of Gaussian densities or any density
proportional to $e^{-p(x)}$ for a positive polynomial $p$.

The scaling limit, introduced in~\cite{RobertsGelmanGilks97}, 
of the chain $\mathbf{X}^d$ as the dimension $d$ tends to infinity is key 
for all that follows. 
Consider a continuous-time process $\{U^d_t\}_{t\geq0}$, given by $U^d_t:=X^d_{\lfloor d\cdot t\rfloor,1}$,
where $\lfloor\cdot\rfloor$ is the integer-part function
and 
$X^d_{\cdot,1}$ is the first coordinate of $\mathbf{X}^d$ (since the proposal distribution for 
$X^d_{\cdot,1}$
has variance  
$l^2/d$,  time needs to be accelerated to get a non-trivial limit). 
As shown in~\cite{RobertsGelmanGilks97} (see also~\cite{RobertsRosenthal01}), the weak convergence 
$U^d\Rightarrow U$ holds as $d\uparrow\infty$,
where $U$ is the Langevin diffusion
started in stationarity, $U_0\sim \rho$, with generator acting on $f\in\mathcal{C}^2(\RR)$ as
\begin{equation}\label{eq:GenU}
\mathcal{G}f:=(h(l)/2)(f''+(\log\rho)'f'), \quad\text{where 
$h(l):=2l^2\Phi\left(-l\sqrt{J}/2\right)$ and $J:=\rho\left((\log(\rho)')^2\right)$}
\end{equation}
and $\Phi$ is the distribution of $N(0,1)$.
Poisson's equation for $U$ and a function $f$ takes the form 
\begin{equation}\label{eq:UfPE}
\mathcal{G}\hat{f}(x)
=\rho(f)-f(x)\text{.}
\end{equation}
It is immediate that a solution $\hat f$ of~\eqref{eq:UfPE} is given by the formula 
\begin{equation}
\label{eq:solutions_Poisson_Eq_diff}
\hat{f}(x):=\int_0^x\frac{2dy}{h(l)\rho(y)}\int_{-\infty}^y{\rho(z)(\rho(f)-f(z))dz},\qquad x\in\RR.
\end{equation}
In the remainder of the paper
$\hat{f}$ denotes the particular solution in~\eqref{eq:solutions_Poisson_Eq_diff}
of the equation in~\eqref{eq:UfPE}.  

As usual, for $p\in[1,\infty)$, $f:\RR^d\to\RR$
is in $L^p(\rho_d)$ if and only if $\|f\|_p:=\rho_d(|f|^p)^{1/p}<\infty$.
Finally, note that under our assumptions on $\rho$, the kernel $P_d$ of the RWM chain
$\mathbf{X}^d$ defined above is a self-adjoint bounded operator on the Hilbert space
$\{g\in L^2(\rho_d)\colon \rho_d(g)=0\}$ with norm 
$\lambda_d<1$ 

\begin{thm}\label{thm:AVbound}
If $f\in\mathcal{S}^1$, then $\hat f\in\mathcal{S}^3$
and CLT~\eqref{eq:CLT} holds for 
the function $f+dP_d\hat{f}-d\hat{f}$ and the RWM chain
$\mathbf{X}^d$ introduced above. 
Furthermore, 
there exists a constant $C_1>0$, such that for all $f\in\mathcal{S}^1$ and $d\in\NN\setminus\{1\}$, 
the asymptotic variance
$\hat\sigma^2_{f,d}$
in CLT~\eqref{eq:CLT}
satisfies:   
$$\hat\sigma^2_{f,d}\leq C_1 \left(\sum_{i=1}^3\|\hat
f^{(i)}\|_{\infty,1/2}\right)^2\frac{\log(d)}{(1-\lambda_d)d}\text{.}$$
\end{thm}

The proof of Theorem~\ref{thm:AVbound} is given in Section~\ref{subsec:proofthm3} below. 
It is based on the spectral-gap estimate of the asymptotic variance 
$\hat\sigma_{f,d}^2\leq 2\|\mathcal{G}_d\hat f-\mathcal{G}\hat f\|^2_2/(1-\lambda_d)$ 
from~\cite{Geyer92,KipnisVaradhan86}, the uniform ergodicity of the chain 
$\mathbf{X}^d$ 
and the following proposition.

\begin{prop}\label{cor:Lpnorm}
There exists a constant $C_2$  such that for every
$f\in\mathcal{S}^3$ and all $d\in\NN\setminus\{1\}$ we have:
$\left\|\mathcal{G}f-{\mathcal{G}}_df\right\|_2\leq
C_2\left(\sum_{i=1}^3\|f^{(i)}\|_{\infty,1/2}\right)\sqrt{\log(d)/d}$.  
\end{prop}

The proof of Proposition~\ref{cor:Lpnorm}, given in Section~\ref{subsec:proofthm2} below,
requires a pointwise control of the difference $\mathcal{G}f-{\mathcal{G}}_df$ on a large subset of $\RR^d$. 
To formulate this precisely, we need the following. 

\begin{defin}\label{def:sluggish}
A positive sequence $a=\{a_d\}_{d\in\NN}$ is \textbf{sluggish} if the following holds:
$$\lim_{d\to\infty} a_d=\infty \quad\text{and}\quad \sup_{d\in\NN\setminus\{1\}}\frac{a_d}{\sqrt{\log{d}}}<\infty\text{.}$$ 
\end{defin}

Theorem~\ref{thm:GdG} below is the main technical result of the paper. 
It generalises the limit in~\cite[Lemma~2.6]{RobertsGelmanGilks97} to a class of unbounded functions
and provides an error estimate for it. 
The bound in Theorem~\ref{thm:GdG} yields sufficient control of the difference 
$\mathcal{G}f-{\mathcal{G}}_df$ to establish Proposition~\ref{cor:Lpnorm}.

\begin{thm}\label{thm:GdG}
Let $a=\{a_d\}_{d\in\NN}$ be a sluggish sequence. There exist 
constants $c_3,C_3>0$ (dependent on $a$)
and measurable sets
$\mathcal{A}_d\subset \RR^d$,
such that for all $d\in\NN$ we have
$\rho_d(\RR^d\setminus \mathcal{A}_d)\leq c_3 e^{-a_d^2}$
and 
$$\left|\mathcal{G}f(x^d_1)-{\mathcal{G}}_df(\mathbf{x}^d)\right|\leq
C_3\left(\sum_{i=1}^3\|f^{(i)}\|_{\infty,1/2}\right) e^{|x^d_1|}
\frac{a_d}{\sqrt{d}} \qquad\text{for any $f\in\mathcal{S}^3$ and $\mathbf{x}^d\in \mathcal{A}_d$.}$$
\end{thm}

The proof of Theorem~\ref{thm:GdG}
is outlined and given in Sections~\ref{subsec:outline} and~\ref{sec:proofthm1} below, respectively.

\begin{rema}
The dependence on $f$ in the bound of Theorem~\ref{thm:GdG} is not sharp. The factor
$\sum_{i=1}^3\|f^{(i)}\|_{\infty,1/2}$ 
is used because it
states concisely that the speed of the convergence of
$\mathcal{G}_df$ to $\mathcal{G}f$ depends linearly on the first three derivatives of
$f$.  
Moreover, it is not clear if 
the bound in Theorem~\ref{thm:AVbound} and
Proposition~\ref{cor:Lpnorm} is optimal in $d$. However, if an improvement were possible, the proof
would have to be significantly different to the one presented here. 
In particular, a better control of the difference
$\left|\mathcal{G}f-{\mathcal{G}}_df\right|$ on $\RR^d\setminus \mathcal{A}_d$ would be
required. 
\end{rema}

\section{Discussion and numerical examples}\label{sec:Extensions}
\subsection{Discussion}

In this section we discuss potential extensions of Theorem~\ref{thm:AVbound} to settings satisfying weaker assumptions or involving related MCMC chains. 

\subsubsection{IID target for non-RWM chains in stationarity}
\label{subsubsec:MALA_fast_MALA}
Perhaps the most natural generalisation of Theorem~\ref{thm:AVbound}
would be to the MALA and fast-MALA chains (in~\cite{MALAScaling} and~\cite{fMALA}) for which it is also possible to obtain non-trivial weak Langevin diffusion limits under appropriate scaling. Since the form of the Poisson equation in~\eqref{eq:UfPE} is preserved it is possible to define an estimator like the one in~\eqref{eq:CLT} and it seems feasible that a version of Theorem~\ref{thm:GdG} can be established in this context using methods analogous to the ones in this paper.


\subsubsection{IID target in the transient phase for the RWM chain}
%

Theorem~\ref{thm:AVbound} is a result only about the stationary behaviour of the chain. As in practice MCMC chains are typically started away from stationarity it is important to understand the transient behaviour. In~\cite{Jourdain_Scaling_RWM}
it is shown that the scaling limit described in Section~\ref{Sec:A&R} above
has mean-field behaviour of the McKean type, i.e. the limiting process is a continuous semimartingale with characteristics 
that at time $t$ depend on the law of the process at $t$. 
This suggests that an appropriately chosen time-dependent function $\hat f$ in 
the estimator in~\eqref{eq:CLT} could further reduce the constant in the bound of Theorem~\ref{thm:AVbound}.

\subsubsection{General product target density}
The class of target distributions considered in~\cite{Bedard2007,BedardRosenthal2008}, 
preserves the independence (i.e. product) structure but allows for a different, dimension dependent, scaling of each of
the components of the target law. If the proposal variances appropriately reflect the scaling in the target, 
each component in the infinite dimensional limit is a Langevin diffusion. Again, as in Section~\ref{subsubsec:MALA_fast_MALA} above, the estimator in~\eqref{eq:CLT} can be applied directly and an extension of 
Theorem~\ref{thm:AVbound} to this setting appears feasible.


\subsubsection{Gaussian targets in high dimensions}\label{subsubsec:Gaussian_targets}
Let $\pi_0$ denote a Gaussian target on $\RR^d$ with mean $\mu$ and covariance matrix
$\text{diag}(\sigma^2_{11},\dots,\sigma^2_{dd})$. Inspired by~\cite[Thm~5]{RobertsRosenthal01}
and the proof of Theorem~\ref{thm:AVbound}, 
a good control variate for the ergodic average estimator for
$\pi_0(f)$
takes the form 
$d(P_d\tilde f-\tilde f)$, 
where $\tilde f$ solves the ODE
$$\tilde{f}''+\left((\partial/\partial x^d_1)\log\pi_0\right)\tilde f'=2/h_0(l)\cdot(\pi_0(f)-f),$$
with $h_0(l):=2l^2\Phi(-l\sqrt{J_0}/2)$ and $J_0:=1/d\cdot \sum_{j=2}^d1/\sigma^2_{jj}$. 
In the case of the mean, $f(x^d_1)=x^d_1$, we can
solve the ODE explicitly: $\tilde f(x^d_1)=2\sigma^2_{11}/h(l)\cdot x^d_1$.

If $\pi_0$ has a general non-degenerate covariance matrix $\Sigma$,
we have an ODE analogous to the one above for each eigen-direction of $\Sigma$. 
The control variate for the mean of the first coordinate,
$f(x^d_1)=x^d_1$, 
is then a linear combination of the 
control variates for the means in eigen-directions. 
Specifically, 
$\tilde f:\RR^d\to\RR$
in the estimator analogous to the one in~\eqref{eq:CLT}
(see the numerical example for $h=0$ in Section~\ref{subsubsec:non-prod-example}) takes the form 
\begin{equation}\label{eq:GaussianMeanCV}
\tilde f(\mathbf{x}^d)=2/h_0(l) \cdot \sum_{j=1}^dx^d_j\Sigma_{j1},\quad \text{with
} h_0(l)=2l^2\Phi(-l\sqrt{J_0}/2) \text{ and } J_0=1/d\cdot
\text{Tr}((\Sigma^{-1})_{2:d,2:d}).  
\end{equation}
Note that 
$\tilde f$
does not depend on the mean of the target, a special feature of the
Gaussian setting.  

\subsubsection{Non-product target density}
\label{subsubsec:Beskos_Roberts_Stuart}
A typical non-product target density 
considered in the literature~\cite{Optimal_Scaling_Var,Optimal_Scaling_RWM_hilbert,Optimal_Scaling_MALA_hilbert}
is a projection of a probability measure $\Pi$ on a separable real Hilbert space $\mathcal{H}$
onto a $d$-dimensional subspace, where $\Pi$ is given via its Radon-Nikodym derivative 
$\frac{d \Pi}{d \Pi_0}(x)\propto \exp(-\Psi(x))$. Here $\Psi$ is a densely defined positive functional on 
$\mathcal{H}$
and  
$\Pi_0$
is a Gaussian measure on $\mathcal{H}$
specified via a positive trace-class operator on $\mathcal{H}$ (see e.g.~\cite[Sec.2.1]{Optimal_Scaling_RWM_hilbert} for a detailed description
and~\cite{Optimal_Scaling_Var} for the motivation for this class of measures).
A key feature of this framework is that there exists an $\mathcal{H}$-valued Langevin diffusion $z$, driven by a cylindrical Brownian motion
on $\mathcal{H}$ (i.e. a solution of an SPDE), that describes the scaling limit of the appropriately accelerated sequence of chains 
$(\mathbf{X}^d)_{d\in\NN}$, see \cite{Optimal_Scaling_RWM_hilbert,Optimal_Scaling_MALA_hilbert}. 

An estimator analogous to the one in~\eqref{eq:CLT}
would require the solution $\hat f$ of the Poisson equation of $z$. 
While this might be a feasible strategy theoretically, it would likely be difficult to numerically evaluate the solution 
$\hat f$. 
However, inspired by the Simplified Langevin Algorithm in~\cite{Optimal_Scaling_Var}, which uses as the proposal chain 
an Euler scheme for the Langevin diffusion with target $\Pi_0$ (not $\Pi$), 
we suggest constructing the estimator 
for $\Pi(f)$ in~\eqref{eq:CLT},
with $\hat f$ the solution of the Poisson equation for the Langevin diffusion converging to  
$\Pi_0$ (not $\Pi$).
This strategy is feasible 
as we are able to produce good control variates for Gaussian product targets
in high dimensions  in the spirit of Theorem~\ref{thm:AVbound},
see Section~\ref{subsubsec:Gaussian_targets} above.

The main theoretical question in this context is to find 
suitable assumptions on the functional $\Psi$ in the Radon-Nikodym derivative above that guarantee the asymptotic variance 
reduction as $d\to\infty$. Expecting an improvement from polynomial to logarithmic growth is unrealistic as we
are solving the Poisson equation for $\Pi_0$ instead of $\Pi$. 
However, the numerical example in Section~\ref{subsubsec:non-prod-example} suggests that this idea may work in practice if $\Pi$ is close from Gaussian $\Pi_0$. Understanding in which settings does it lead to significant variance reduction is another relevant and important question.

\subsection{Numerical examples}\label{sec:Implementation}

The basic message of the present paper is that the process in the scaling limit
of an MCMC algorithm contains useful information that can be utilised to achieve significant savings in 
high dimensions. 


In both examples presented below the problem is to estimate the mean of the
first coordinate $\rho_d(f)$, where $f(\mathbf{x}^d):=x^d_1$. 
A run of $T$ steps of a well-tuned RWM algorithm with kernel
$P_d$, defined in the beginning of Section~\ref{Sec:A&R},
started in stationarity, produces a RWM sample
$\{\mathbf{X}^d_n\}_{n=1,2,\cdots T}$ and an estimate
$\hat{\rho}_d(f):=\sum_{n=1}^{T}f(\mathbf{X}^d_n)/T$ of $\rho_d(f)$. Take $\tilde{f}$ to be the associated solution of the Poisson equation, that we obtain numerically in the first example~\ref{subsubsec:Numerical_example_multi_mod} and using formula \eqref{eq:GaussianMeanCV} in the second example~\ref{subsubsec:non-prod-example}. In both cases we estimate the required unknown quantities ($\rho_d(f)$ and $\Sigma$) from the sample $\{\mathbf{X}^d_n\}_{n=1,2,\cdots T}$.

Using the same sample, 
define
$\tilde{\rho}_d(f):=\frac{1}{T}\sum_{n=1}^{T}(f+dP_d\tilde{f}-d\tilde{f})(\mathbf{X}^d_n)$.
Since the function 
$P_d\tilde{f}-\tilde{f}$
is not accessible in closed form,
for every $n\leq T$ we use IID Monte Carlo to estimate the value
$(P_d\tilde{f}-\tilde{f})(\mathbf{X}^d_n)$  
as
$\sum_{j=1}^{n_{MC}}\left(1\wedge
(\rho_d(\mathbf{Y}^{d,j}_n)/\rho_d(\mathbf{X}^d_n))\right)\left(\tilde{f}(\mathbf{Y}^{d,j}_n)-\tilde{f}(\mathbf{X}^d_n)\right)/n_{MC}$,
where
$\mathbf{Y}^{d,1}_n,\dots,\mathbf{Y}^{d,n_{MC}}_n$ is an IID sample 
of size $n_{MC}$ from $N(\mathbf{X}^d_n,l^2/d\cdot I_d)$. This estimation step can be parallelised (i.e. run on $n_{MC}$ cores simultaneously).

We measure the variance reduction due to the post processing above by comparing the mean 
square errors of $\hat{\rho}_d(f)$ and
$\tilde{\rho}_d(f)$ as estimators of $\rho_d(f)$ over $n_R$
independent runs of the RWM chain, 
\begin{equation}
\label{eq:EVR}
\text{VR}(\rho,f):=\frac{\sum_{k=1}^{n_R}((\hat{\rho}_d(f))_k-\rho_d(f))^2}{\sum_{k=1}^{n_R}((\tilde{\rho}_d(f))_k-\rho_d(f))^2}\text{,}
\end{equation}
where $(\hat{\rho}_d(f))_k$ and $(\tilde{\rho}_d(f))_k$ are the averages in the $k$-th run of the chain. Heuristically, this means the estimator $\tilde{\rho}_d(f)$ of $\rho_d(f)$ based on $T$ sample points is as good as estimator  $\hat{\rho}_d(f)$ based on $\text{VR}(\rho,f)\cdot T$ sample points.

\subsubsection{Multi-modal product target}
\label{subsubsec:Numerical_example_multi_mod}
To verify that what theory predicts also happens in practice we first present an example of an IID target with each coordinate a bimodal mixture of two Gaussian densities. Given the results in Table~\ref{table:1}, we wish to highlight the robustness of the method with respect to numerically estimating $\tilde{f}$ and $(P_d\tilde{f}-\tilde{f})(\mathbf{X}^d_n)$  for each $n\leq T$.

Let $\rho$ be a mixture of two normal densities 
$N(\mu_1,\sigma_1^2)$ and 
$N(\mu_2,\sigma_2^2)$, with the first arising in the mixture with probability $2/5$
and
$\mu_1=-3, \mu_2=4$ and $\sigma_1=\sigma_2=7/4$. 
The potential of the density $\rho$ has two wells and is in a well know
class arising in models of molecular dynamics, see e.g.~\cite[Sec.~5.4]{Pavliotis}. 
Note that the corresponding density 
$\rho_d$ on $\RR^d$, defined in the introduction, has 
$2^{d}$ modes. 

With variance reduction~\eqref{eq:EVR} we measure how much the estimator $\tilde{\rho}_d(f)$ outperforms the estimator $\hat{\rho}_d(f)$ of the mean of the first coordinate $\rho_d(f)=6/5$. The results across a range of dimensions $d$ and values of the parameter $n_{MC}$ (that corresponds to the accuracy of the estimation of $P_d\tilde{f}-\tilde{f}$) are presented in Table~\ref{table:1}. All the entries were computed using $n_R=500$ independent runs of length $T=2\cdot 10^5$.


\begin{table}[ht]
\begin{center}
\begin{tabular}{|c||c|c|c|c|c|c|c|}
\hline
$d\backslash n_{MC}$  & 30  & 50  & 70 & 150 & 300 \\ \hline\hline
$5$   &  $17.5$    &  $20.5$  &  $23.9$  & $28.4$  &  $28.9$  \\ \hline
$10$  &  $20.9$    &  $36.5$  &  $40.4$  & $54.7$  &  $78.9$  \\ \hline
$20$  &  $31.4$    &  $46.4$  &  $65.0$  & $105.6$ &  $116.4$   \\ \hline
$30$  &  $35.5$    &  $51.3$  &  $71.0$  & $127.1$ &  $163.1$   \\ \hline
$50$  &  $35.7$    &  $55.7$  &  $77.8$  & $136.8$ &  $196.8$   \\ \hline
\end{tabular}
\end{center}
\caption{Variance reduction for different dimensions $d$ and values of $n_{MC}$.}
\label{table:1}
\end{table}

To numerically solve the Poisson equation in~\eqref{eq:UfPE}, substitute the
derivatives of $f$ and $\log(\rho)$ with symmetric finite differences and use the estimate
$\hat{\rho}(f)$ for $\rho(f)$. 
Recall that the standard deviation of the proposal in our RWM algorithm is
$l/\sqrt{d}$.
The solver uses a grid of hundred points equally 
spaced 
in the interval 
$[\min_{n\leq T }\mathbf{X}^d_{n,1}- 3 \cdot l/\sqrt{d}, \max_{n\leq T}\mathbf{X}_{n,1}+3\cdot l/\sqrt{d}]$,
where $\mathbf{X}^d_{n,1}$ is the first coordinate of the $n$-th sample point. 
We use the Moore-Penrose pseudoinverse as the linear system is not of full rank. 
Finally, we take $\tilde{f}$ to be the linear interpolation of the solution on the grid.
Note that this is a crude approximation of the solution of~\eqref{eq:UfPE}, which does not exploit
analytical properties of either $f$ or $\rho$.

The results in Table~\ref{table:1} contain a lot of noise due to numerically solving the ODE, using an approximation $\hat{\rho}_d(f)$ for $\rho_d(f)$ and using IID Monte Carlo to estimate $\tilde{\rho}_d(f)$. It is interesting to note, that despite these additional sources of error, the variance reduction is considerable and behaves as the theoretical results predict. The estimator $\tilde{\rho}_d(f)$ improves with dimension, and with $n_{MC}$. Increasing $n_{MC}$, however, has diminishing effect which is particularly clear in the case $d=5$. Due to the asymptotic nature of our result we can only expect limited gain for any fixed $d$, even if $P_d\tilde{f}-\tilde{f}$ could be evaluated exactly (corresponding to $n_{MC}=\infty$).

\subsubsection{Bi-modal non-product target}
\label{subsubsec:non-prod-example}
Can the theoretical findings of this paper help us construct control variates in more realistic cases with non-product target densities? It is unreasonable to expect a simple general answer to this question. A more realistic approach for future work seems to be trying to establish specific forms of control variates that work well for classes of targets of certain type. We briefly explore one such instance in this section. Sections~\ref{subsubsec:Gaussian_targets}~and~\ref{subsubsec:Beskos_Roberts_Stuart} as well as the results in Table~\ref{table:2} suggest we can construct useful control variates when the target is close to a Gaussian. 

Let $\mu_{d,h}$ be a $d$-dimensional vector with entries $(h/2,0,\dots,0)$ for $h\geq 0$ and let $\Sigma^{(d)}$ be a $d\times d$ covariance matrix with the largest eigenvalue equal to $\lambda=25$ with the corresponding eigenvector being $(1,1\dots 1)$ and all other eigenvalues being equal to one. 
Take $\Pi_{d,h}$ to be the mixture of two $d$-dimensional normal densities 
$N(-\mu_{d,h},\Sigma^{(d)})$ and 
$N(\mu_{d,h},\Sigma^{(d)})$, both arising in the mixture with probability $1/2$. 

We wish to estimate the mean of the first coordinate $\Pi_{d,h}(f)=0$ (for $f(\mathbf{x}^d)=x^d_1$). To produce a control variate we simply pretend, that we are dealing with a Gaussian target instead of $\Pi_{d,h}$. Let $\Sigma^{\Pi_{d,h}}$ be the covariance of $\Pi_{d,h}$ and $\hat\Sigma^{\Pi_{d,h}}$ an estimate of $\Sigma^{\Pi_{d,h}}$ obtained from the RWM sample $\{\textbf{X}^d_n\}_{n=1,2,\dots,T}$. Define $\tilde{f}_{d,h}$ as in \eqref{eq:GaussianMeanCV} using $\hat\Sigma^{\Pi_{d,h}}$. We compare the performance of estimators $\hat{\Pi}_{d,h}(f)$ and $\tilde{\Pi}_{d,h}(f)$ of  $\Pi_{d,h}(f)=0$, respectively defined as $1/T\sum_{n=1}^Tf(\textbf{X}^d_n)$ and $1/T\sum_{n=1}^T\left(f+dP_d\tilde{f}_{d,h}-d\tilde{f}_{d,h}\right)(\textbf{X}^d_n)$, according to variance reduction~\eqref{eq:EVR}.

Table~\ref{table:2} shows the results across a range of dimensions $d$ and distances between modes $h$, which measures the 'non-Gaussianity' of the target. Note that when $h=0$ the target is Gaussian $N(0,\Sigma^{(d)})$ which we include to demonstrate the validity of control variate \eqref{eq:GaussianMeanCV} for Gaussian targets. All the entries were calculated using $n_R=500$ independent runs of length $T=2\cdot 10^5$ and $n_{MC}=50$ IID Monte Carlo steps for computing $P_d\tilde{f}-\tilde{f}$ at each time step.

\begin{table}[ht]
\begin{center}
\begin{tabular}{|c||c|c|c|c|c|c|}
\hline
$d\backslash h$  & 0  & 2  & 4 & 6 & 8 & 10\\ \hline\hline
$5$   &  $60.1$    &  $40.5$  &  $34.6$  & $3.78$  &  $1.21$ & $1.01$	\\ \hline
$10$  &  $59.3$    &  $38.2$  &  $12.4$  & $1.88$  &  $1.05$ & $1.00$	\\ \hline
$20$  &  $46.8$    &  $37.6$  &  $7.00$  & $1.51$  &  $1.01$ & $1.00$	\\ \hline
$30$  &  $37.7$    &  $36.2$  &  $5.88$  & $1.31$  &  $1.01$ & $1.00$	\\ \hline
$50$  &  $27.8$    &  $25.8$  &  $3.50$  & $1.26$  &  $1.00$ & $1.00$	\\ \hline
\end{tabular}
\end{center}
\caption{Variance reduction for different dimensions $d$ and distances between modes $h$.}
\label{table:2}
\end{table}

The quality of results decays with dimension because the proposal is scaled as $1/d$ in each coordinate. This results in the first coordinate mixing slower and for $h\neq 0$ also being less able to cross between modes, hence our estimate $\hat\Sigma^{\Pi_{d,h}}$ of the covariance becomes worse as we are working with fixed RWM sample length $T$. When $h=10$ (and some cases of $h=8$) it is unlikely that the RWM sample will reach the other mode at all which results in no gain from the method.

If we use the true covariance $\Sigma^{\Pi_{d,h}}$ of the target in the control variate \eqref{eq:GaussianMeanCV}, instead of learning it from the sample $\hat\Sigma^{\Pi_{d,h}}$, the corresponding results for $d=50$ are presented in Table~\ref{table:3}.

\begin{table}[ht]
\begin{center}
\begin{tabular}{|c||c|c|c|c|c|c|}
\hline
$h$  & 0  & 2  & 4 & 6 & 8 & 10\\ \hline\hline
$d=50$  &  $73.0$    &  $59.2$  &  $3.91$  & $1.31$  &  $1.02$ & $0.99$	\\ \hline
\end{tabular}
\end{center}
\caption{Variance reduction for dimension $d=50$ and different distances between modes $h$ using the true covariance of the target.}
\label{table:3}
\end{table}

Unsurprisingly the estimator $\tilde{\Pi}_{d,h}(f)$ does not perform well when the distance between modes $h$ is large. Interestingly though, the method does offer considerable gain in cases $h=2$ and $h=4$, even a noticeable gain in $h=6$. For $h=4$ and $h=6$ the target is already clearly bimodal and different from the Gaussian, the RWM sample stays in the same mode for hundreds, respectively thousands of time-steps at a time.

\section{Proofs}\label{sec:proofs}
Throughout this section we assume the sluggish sequence $a=\{a_d\}_{d\in\NN}$
is given and fixed and, as mentioned above, the density $\rho$ satisfies $\log(\rho)\in\mathcal{S}^4$
and has sub-exponential tails~\eqref{eq:SupperExponentialTails}.
Section~\ref{subsec:outline}
outlines the proof of Theorem~\ref{thm:GdG} by 
stating the sequence of results that are needed to establish it. 
The proofs of these results, given in Section~\ref{sec:proofthm1}, rely on the  
theory developed in Section~\ref{sec:Technical} below. 
Sections~\ref{subsec:proofthm2} and~\ref{subsec:proofthm3} establish 
Proposition~\ref{cor:Lpnorm} and Theorem~\ref{thm:AVbound}, respectively.

\subsection{Outline of the proof of Theorem~\ref{thm:GdG}}
\label{subsec:outline}

We start by specifying sets $\mathcal{A}_d\subset \RR^d$ that have large probability under $\rho_d$. We need the following fact.

\begin{prop}\label{prop:auxAndef}
There exists a constant 
$c_A>0$, such that  
the following open subset of $\RR$, 
$A:=\{x\in\RR; |\log(\rho)''(x)|< (\log(\rho)'(x))^2,1/c_A<|\log(\rho)'(x)|< c_A\}$,
satisfies $\rho(A)>0$.
\end{prop}

Let $A$ satisfy the conclusion of Proposition~\ref{prop:auxAndef} and 
recall the notation $\rho(f)=\int_{\RR}f(x)\rho(x) dx$ for any appropriate function $f:\RR\to\RR$.
Recall
$J=\rho(((\log\rho)')^2)=-\rho((\log \rho)'')$,
where the equality follows from assumptions $\log(\rho)\in\mathcal{S}^4$ and~\eqref{eq:SupperExponentialTails}.
Define the sets $\mathcal{A}_d$ as follows.

\begin{defin}\label{def:cond}
Any $\mathbf{x}^d\in\RR^d$ is in $\mathcal{A}_d$ 
if and only if 
the following four assumptions hold: 
\begin{eqnarray}
&&\frac{1}{d-1}\sum_{i=2}^de^{|x^d_i|}< 2\int_\RR e^{|x|}\rho(x)dx\text{,}\label{eq:AnCond1}\\
&&\frac{1}{d-1}\sum_{i=2}^d1_{A}(x^d_i)> \frac{\rho(A)}{2}\text{,}\label{eq:AnCond2}\\
&&\frac{1}{d-1}\left|\sum_{i=2}^d \left(\log(\rho)'(x^d_i)\right)^2-J\right|<\frac{a_d}{\sqrt{d}} 
\sqrt{3\int_{\RR}\left(\left(\log(\rho)'(x)\right)^2-J\right)^2\rho(x)dx}\text{,}\label{eq:AnCond3}\\
&&\frac{1}{d-1}\left|\sum_{i=2}^d \log(\rho)''(x^d_i)+J\right|<\frac{a_d}{\sqrt{d}} 
\sqrt{3\int_{\RR}\left(\log(\rho)''(x)+J\right)^2\rho(x)dx}\label{eq:AnCond4}\text{.}
\end{eqnarray}
\end{defin}

\begin{rema}
The precise form of the constants in Definition~\ref{def:cond} is chosen purely for convenience.
It is important that $\int_\RR e^{|x|}\rho(x)dx<\infty$
by~\eqref{eq:SupperExponentialTails}, $\rho(A)>0$ by Proposition~\ref{prop:auxAndef} and that
the constants in~\eqref{eq:AnCond3}--\eqref{eq:AnCond4} are in $(0,\infty)$.
Moreover, 
for any $\mathbf{x}^d\in\mathcal{A}_d$
there are no restrictions on its first coordinate $x_1^d$ 
and
the sets $\mathcal{A}_d$ are typical in the following sense. 
\end{rema}

\begin{prop}\label{prop:AdSizes}
There exists a constant $c_1$, such that $\rho_d(\RR^d\setminus \mathcal{A}_d)\leq c_1 e^{-a_d^2}$ for all $d\in\NN$.
\end{prop}
Using the theory of large deviations and classical inequalities,  
the proof of the proposition bounds the probabilities of sets where each of the above four assumptions 
in Definition~\ref{def:cond} fails 
(see Sections~\ref{sec:proofthm1}
and~\ref{sec:EventsBounds} below for details).

Pick any $f\in\mathcal{S}^3$
and express the generator $\mathcal{G}_d$, defined in~\eqref{eq:MetropolisGenerator}, 
 as follows: 
$$\mathcal{G}_df(\mathbf{x}^d)=d\cdot
\mathbb{E}_{Y^d_1}\left[\left(f(\mathbf{Y}^d)-f(\mathbf{x}^d)\right)
\mathbb{E}_{\mathbf{Y}^d-}\left[1\wedge\frac{\rho_d(\mathbf{Y}^d)}{\rho_d(\mathbf{x}^d)}\right]\right],
\qquad\mathbf{x}^d\in\RR^d, $$
where 
$\mathbb{E}_{\mathbf{Y}^d-}\left[\cdot\right]$
is the expectation with respect to all the coordinates of the proposal 
$\mathbf{Y}^d$ in $\RR^d$, except the first one 
(identify $f\in L^1(\rho)$ with $f\in L^1(\rho_d)$ by ignoring the last $d-1$ coordinates).
The strategy of the proof of Theorem~\ref{thm:GdG} is to define a
sequence of operators, ``connecting'' $\mathcal{G}_d$ and $\mathcal{G}$,
such that each approximation can be controlled for $f\in\mathcal{S}^3$ and
$\mathbf{x}^d\in\mathcal{A}_d$.

First, for $f\in\mathcal{S}^3$, define 
\begin{equation}\label{eq:GtildaGenerator}
\tilde{\mathcal{G}}_df(\mathbf{x}^d):=d\cdot
\mathbb{E}_{Y^d_1}\left[\left(f(Y^d_1)-f(x^d_1)\right)\beta(\mathbf{x}^d,Y^d_1)\right],\qquad \mathbf{x}^d\in\RR^d\text{,}
\end{equation}
where for any $y\in\RR$,
\begin{equation}\label{eq:betadef}
\beta(\mathbf{x}^d, y):=\mathbb{E}_{\mathbf{Y}^d-}
\left[1\wedge \exp\left(\log(\rho)'(x^d_1)(y-x^d_1)+\sum_{i=2}^d K(x^d_i, Y^d_i)\right)\right],\qquad\mathbf{x}^d\in\RR^d,
\end{equation}
and for any $(x,y)\in\RR^2$ we define
\begin{equation}\label{eq:Mdef}
K(x, y):=\log(\rho)'(x)(y-x)+\frac{\log(\rho)''(x)}{2}(y-x)^2+\frac{\left(\log(\rho)'(x)\right)^31_{A}(x)}{3}(y-x)^3\text{.}
\end{equation}

In~\eqref{eq:Mdef}, the set $A$ satisfies the conclusion of Proposition~\ref{prop:auxAndef}
and the coefficient before $(y-x)^3$ is chosen so that it is uniformly bounded 
for all $x\in \RR$. This property plays an important role in proving that we have uniform control over the supremum norms of certain
densities, cf. Lemmas~\ref{lemma:Mpdf} and~\ref{lemma:sumMpdf} below.
We can now prove the following.

\begin{prop}\label{prop:GdTildaG} There exists a constant $C$, such that for every $f\in\mathcal{S}^3$ and all $d\in\NN$ we have:
$$\left|\mathcal{G}_df(\mathbf{x}^d)-\tilde{\mathcal{G}}_df(\mathbf{x}^d)\right|\leq C \|f'\|_{\infty,1/2} e^{|x^d_1|}d^{-1/2}
\quad \forall\mathbf{x}^d\in\mathcal{A}_d\text{.} $$
\end{prop}

The proof of Proposition~\ref{prop:GdTildaG}
relies only on the elementary bounds from Section~\ref{sec:TestFunBounds} below. 
The idea is to use the Taylor series of $\log(\rho)(Y_i)$ around $x^d_i$ for every $i\in\{1,\ldots,d\}$ 
and then prove 
that modifying terms of order higher then two if $i\in\{2,\ldots,d\}$
(resp. one if $i=1$) is inconsequential.

Define the operator 
$\hat{\mathcal{G}}_df(\mathbf{x}^d)$ for any $f\in\mathcal{S}^3$ and $\mathbf{x}^d\in\RR^d$ by
\begin{eqnarray}
\hat{\mathcal{G}}_df(\mathbf{x}^d)&:=&\frac{l^2}{2}f''(x^d_1)\mathbb{E}_{\mathbf{Y}^d-}\left[1\wedge
e^{\sum_{i=2}^d K(x^d_i, Y^d_i)}\right]\label{eq:GhatGenerator}\\
&+&l^2f'(x^d_1)\log(\rho)'(x^d_1)\mathbb{E}_{\mathbf{Y}^d-}\left[e^{\sum_{i=2}^d
K(x^d_i, Y^d_i)}1_{\left\{\sum_{i=2}^d K(x^d_i,
Y^d_i)<0\right\}}\right]\text{.}\nonumber 
\end{eqnarray}
We can now prove the following fact.  

\begin{prop}\label{prop:TildaGHatG}
There exists a constant $C$, such that for every $f\in\mathcal{S}^3$, 
and all $d\in\NN$ we have:
$$\left|\hat{\mathcal{G}}_df(\mathbf{x}^d)-\tilde{\mathcal{G}}_df(\mathbf{x}^d)\right|\leq
C\left(\sum_{i=1}^3\|f^{(i)}\|_{\infty,1/2}\right)e^{|x^d_1|}d^{-1/2}\quad \forall\mathbf{x}^d\in\mathcal{A}_d\text{.} 
$$ 
\end{prop}

Note that, if we freeze the coordinates $x_2^d,\ldots,x_d^d$ in $\mathbf{x}^d$,
the operator mapping $f\in\mathcal{S}^3$ to  
$x_1^d\mapsto \hat{\mathcal{G}}_df(\mathbf{x}^d)$ 
generates a one-dimensional diffusion with coefficients of the same functional form as in 
$\mathcal{G}$, but with slightly modified parameter values. 
The proof of Proposition~\ref{prop:TildaGHatG} 
is based on the third and second degree Taylor's expansion of $y\mapsto f(y)$  and
 $y\mapsto \beta(\mathbf{x}^d,y)$ (around $x^d_1$), respectively, 
applied to the definition of $\tilde{\mathcal{G}}_d$ in~\eqref{eq:GtildaGenerator}.
The difficult part in proving that the remainder terms can be omitted
consist of controlling
$\frac{\partial^2}{\partial y^2}\beta(\mathbf{x}^d,y)$, as this entails 
bounding the supremum norm of the density of $\sum_{i=2}^d K(x^d_i, Y^d_i)$ uniformly in $d$. 
Condition~\eqref{eq:AnCond2}, which forces a portion of the coordinates $x^d_i$ of
$\mathbf{x}^d$ to be in the set $A$ where the densities of the corresponding
summands $K(x^d_i,Y^d_i)$ can be controlled, was introduced for this purpose. The details, explained in
Sections~\ref{sec:proofthm1} and~\ref{sec:PDFBOunds} below, rely crucially on the optimal version of Young's 
inequality.

Introduce the following 
normal random variable with mean
$\mu_{\mathcal{N}}(\mathbf{x}^d)=\frac{l^2}{2d}\sum_{i=2}^d\log(\rho)''(x^d_i)$
and variance
$\sigma^2_{\mathcal{N}}(\mathbf{x}^d)=\frac{l^2}{d}\sum_{i=2}^d\left(\log(\rho)'(x^d_i)\right)^2$:
\begin{equation}\label{eq:Ndef}
\mathcal{N}(\mathbf{x}^d,\mathbf{Y}^d):=\frac{l^2}{2d}\sum_{i=2}^d\log(\rho)''(x^d_i)+\sum_{i=2}^d\log(\rho)'(x^d_i)(Y^d_i-x^d_i)\text{.}
\end{equation}
Define the operator 
$\breve{\mathcal{G}}_df(\mathbf{x}^d)$
for $f\in\mathcal{S}^3$ and
$\mathbf{x}^d\in\RR^d$ by: 
\begin{eqnarray}
\breve{\mathcal{G}}_df(\mathbf{x}^d)&:=&\frac{l^2}{2}f''(x^d_1)\mathbb{E}_{\mathbf{Y}^d-}\left[1\wedge e^{\mathcal{N}(\mathbf{x}^d,\mathbf{Y}^d)}\right]\label{eq:GbreveGenerator}\\
&+&l^2f'(x^d_1)\log(\rho)'(x^d_1)\mathbb{E}_{\mathbf{Y}^d-}\left[e^{\mathcal{N}(\mathbf{x}^d,\mathbf{Y}^d)}1_{\left\{\mathcal{N}(\mathbf{x}^d,\mathbf{Y}^d)<0\right\}}\right]\text{.}\nonumber
\end{eqnarray}

\begin{prop}\label{prop:HatGBreveG}
There exists a constant $C$, such that for every $f\in\mathcal{S}^3$ and all $d\in\NN$ we have:
$$\left|\breve{\mathcal{G}}_df(\mathbf{x}^d)-\hat{\mathcal{G}}_df(\mathbf{x}^d)\right|\leq C\left(\sum_{i=1}^2\|f^{(i)}\|_{\infty,1/2}\right) e^{|x^d_1|}d^{-1/2} \quad\forall\mathbf{x}^d\in\mathcal{A}_d\text{.} $$
\end{prop}

First we show that $|\mathbb{E}_{\mathbf{Y}^d-}[1\wedge e^{\sum_{i=2}^d
K(x^d_i, Y^d_i)}]-\mathbb{E}_{\mathbf{Y}^d-}[1\wedge
e^{\mathcal{N}(\mathbf{x}^d,\mathbf{Y}^d)}]|$ is small (Lemma~\ref{lemma:MNdifference} below). 
Proving that 
$\mathbb{E}_{\mathbf{Y}^d-}\left[e^{\sum_{i=2}^d K(x^d_i,
Y^d_i)}1_{\left\{\sum_{i=2}^d K(x^d_i, Y^d_i)<0\right\}}\right]$ and
$\mathbb{E}_{\mathbf{Y}^d-}\left[e^{\mathcal{N}(\mathbf{x}^d,\mathbf{Y}^d)}1_{\left\{\mathcal{N}(\mathbf{x}^d,\mathbf{Y}^d)<0\right\}}\right]$
are close is challenging, as it requires showing that the supremum norm of the difference
between the distributions of $\mathcal{N}(\mathbf{x}^d,\mathbf{Y}^d)$
and $\sum_{i=2}^d K(x^d_i, Y^d_i)$ decays as $d^{-1/2}$ uniformly in its argument. 
The proof of this fact mimics the proof of the Berry-Esseen theorem and relies on the closeness of the
CFs (characteristic functions) of $\mathcal{N}(\mathbf{x}^d,\mathbf{Y}^d)$ and
$\sum_{i=2}^d K(x^d_i, Y^d_i)$. The particular form of $K(x,Y)$ makes it
possible to explicitly calculate the CF of $K(x,Y)$, if $x\notin A$, and bound
it appropriately, if $x\in A$. The details are explained
in Sections~\ref{sec:proofthm1} and~\ref{sec:CDFandCFBounds} below.

Since $\mathcal{N}(\mathbf{x}^d,\mathbf{Y}^d)$ is normal,  it is
possible to explicitly calculate the expectations
$\mathbb{E}_{\mathbf{Y}^d-}\left[1\wedge e^{\mathcal{N}(\mathbf{x}^d,\mathbf{Y}^d)}\right]$ and
$\mathbb{E}_{\mathbf{Y}^d-}\left[e^{\mathcal{N}(\mathbf{x}^d,\mathbf{Y}^d)}1_{\left\{\mathcal{N}(\mathbf{x}^d,\mathbf{Y}^d)<0\right\}}\right]$,
see~\cite[Prop.~2.4]{RobertsGelmanGilks97}.
Using these formulae, 
Proposition~\ref{prop:BreveGG}, which implies Theorem~\ref{thm:GdG},
can be deduced from  assumptions~\eqref{eq:AnCond3}--\eqref{eq:AnCond4}. 

\begin{prop}\label{prop:BreveGG}
There exists a constant $C$, such that for every $f\in\mathcal{S}^3$ and all $d\in\NN$ we have:
$$\left|\mathcal{G}f(x^d_1)-\breve{\mathcal{G}}_df(\mathbf{x}^d)\right|\leq C\left(\sum_{i=1}^2\|f^{(i)}\|_{\infty,1/2}\right)  e^{|x^d_1|}\frac{a_d}{\sqrt{d}} \quad\forall\mathbf{x}^d\in\mathcal{A}_d\text{.}$$
\end{prop}

\begin{rema}
The bounds in Propositions~\ref{prop:GdTildaG}--\ref{prop:HatGBreveG} 
are of the order $\O(d^{-1/2})$. The order  $\O(a_d/\sqrt{d})$ of the bound in Proposition~\ref{prop:BreveGG} 
gives the order in the bound of Theorem~\ref{thm:GdG}.
\end{rema}

\subsection{Proof of Theorem~\ref{thm:GdG}}\label{sec:proofthm1}
\begin{proof}[Proof of Proposition~\ref{prop:auxAndef}]
Let $\tilde{A}:=\left\{ x\in\RR;\left|\log(\rho)''(x)\right|< \left(\log(\rho)'(x)\right)^2\right\}$.
It suffices to show that the open set $\tilde A$
is not empty, since 
$\tilde{A}=\cup_{n\in\NN}(\tilde{A}\cap\left\{x\in\RR; \frac{1}{n}< \left|\log(\rho)'(x)\right|<n\right\})$,
so for some large $n_0$ the open set 
$\tilde{A}\cap\left\{x\in\RR; \frac{1}{n_0}<\left|\log(\rho)'(x)\right|< n_0\right\}$ must have positive Lebesgue measure and we can take $c_A:=n_0$.

Assume that $\tilde{A}=\emptyset$, i.e. 
$|u'|\geq u^2$ 
on
$\RR$, 
where
$u:=\log(\rho)'$.
Since $\rho$ satisfies \eqref{eq:SupperExponentialTails},
there exists $x_0<0$ and $C>0$ such that $u>C$ 
on the interval $(-\infty,x_0)$. Moreover,  since
$|u'|\geq u^2>C^2>0$,
$u'$ has no zeros on $(-\infty,x_0)$ and satisfies
either
$u'\geq u^2$  or $-u'\geq u^2$
on the half-infinite interval. 
Since $(1/u)'=-u'/u^2$, integrating the inequalities
$-u'/u^2\leq -1$  or $-u'/u^2\geq1$ from  any
$x\in(-\infty,x_0)$ to $x_0$,
we get
$1/u(x_0)+x_0-x\leq 1/u(x)$
and 
$1/u(x_0)+x-x_0\geq 1/u(x)$.
Since by assumption it holds
$0<1/u<1/C$
on $(-\infty,x_0)$, we get a contradiction in both cases. 
\end{proof}

\begin{proof}[Proof of Proposition~\ref{prop:AdSizes}]\label{proof:AdSizes}
Let $B^d_1$, $B^d_2$, $B^d_3$ and $B^d_4$ be the subsets of
$\RR^d$ where assumptions \eqref{eq:AnCond1}, \eqref{eq:AnCond2},
\eqref{eq:AnCond3} and \eqref{eq:AnCond4} are not satisfied, respectively. Note that
$\RR^d\setminus \mathcal{A}_d=B^d_1\cup B^d_2\cup B^d_3 \cup B^d_4$.

Recall that by~\eqref{eq:SupperExponentialTails}, the L'Hospital's rule implies $\lim_{|x|\to\infty}\frac{\log\rho(x)}{x}\to-\infty$ and hence
$\rho(e^{s|x|})<\infty$ for any $s>0$.
Since $\{a_d\}_{d\in\NN}$ is sluggish, there exists 
$n\in\NN$
such that 
$a_d\leq  \sqrt{n\log d}$ for all $d\in\NN$.
Then, by Proposition~\ref{prop:AdSizesAux} applied to 
functions $x\mapsto (e^{|x|}-\rho(e^{|x|}))/\rho(e^{|x|})$ and $x\mapsto 2(\rho(A)-1_A(x))/\rho(A)$, respectively,
there
exist constants $c_1', c_2'$ such that the inequalities 
$\rho(B^d_1)\leq c_1'd^{-n}\leq c_1'e^{-a_d^2}$ and $\rho(B^d_2)\leq c_2'd^{-n}\leq c_2'e^{-a_d^2}$ hold for all $d\in\NN$. 

Likewise, there exist constants $c_3', c_4'$ such that $\rho(B^d_3)\leq c_3'
e^{-a^2_d}$ and $\rho(B^d_4)\leq c_4' e^{-a^2_d}$. This follows by
Proposition~\ref{prop:Probabilities&Sets}, applied to the sequence
$\{a_d\}_{d\in\NN}$ and functions $g_3(x):=\left(\log(\rho)'(x)\right)^2-J$ (with
$t:=\sqrt{3\rho(g_3^2)}$)
and $g_4(x):=\log(\rho)''(x)+J$ (with
$t:=\sqrt{3\rho(g_4^2)}$), respectively.
Hence
$\rho(\RR^d\setminus \mathcal{A}_d)\leq \rho(B^d_1)+\rho(B^d_2)+\rho(B^d_3 )+\rho(B^d_4)\leq
c_1 e^{-a_d^2}$ for $c_1:=\max\{c_1',c_2',c_3',c_4'\}$.  
\end{proof}

\begin{rema}
The proof above shows that 
the subsets $B^d_1$ and $B^d_2$ are of negligible
size in comparison to $B^d_3$ and $B^d_4$, since the $n\in\NN$ can be chosen arbitrarily large.
\end{rema}

\begin{proof}[Proof of Proposition~\ref{prop:GdTildaG}]
Pick an arbitrary $\mathbf{x}^d\in \mathcal{A}_d$
and recall that $\alpha(\mathbf{x}^d,\mathbf{Y}^d)$ is defined in~\eqref{eq:MetropolisGenerator}. 
Since $|1\wedge e^x-1\wedge e^y|\leq |x-y|$ for all $x,y\in\RR$, 
for every realization $Y^d_1$,
Taylor's theorem implies 
\begin{equation}
\label{eq:T_1_and_T_2_estimate}
\left|\mathbb{E}_{\mathbf{Y}^d-}[\alpha(\mathbf{x}^d,\mathbf{Y}^d)]-\beta(\mathbf{x}^d,
Y^d_1)\right|\leq
|\log(\rho)''(W^d_1)|(Y^d_1-x^d_1)^2+T^d_1(\mathbf{x}^d)+T^d_2(\mathbf{x}^d)\text{,}
\end{equation}
where $W^d_1$ satisfies  
$\log(\rho)''(W^d_1)(Y^d_1-x^d_1)^2/2=\log(\rho)(Y^d_1)-\log(\rho)(x^d_1)-\log{\rho}'(x^d_1)(Y^d_1-x^d_1)$
and
\begin{eqnarray}
T^d_1(\mathbf{x}^d)&:=&\frac{1}{6}\mathbb{E}_{\mathbf{Y}^d-}\left[\left|\sum_{i=2}^d\left(\log(\rho)'''(x^d_i)-2(\log(\rho)'(x^d_i))^31_{A}(x^d_i)\right)(Y_i-x^d_i)^3\right|\right]\text{,}\nonumber\\
T^d_2(\mathbf{x}^d)&:=&\frac{1}{24}\mathbb{E}_{\mathbf{Y}^d-}\left[\sum_{i=2}^d|\log(\rho)^{(4)}(Z^d_i)|(Y_i-x^d_i)^4\right]\text{.}\nonumber
\end{eqnarray}
Here $Z^d_i$ satisfies 
$\log(\rho)^{(4)}(Z^d_i)(Y^d_i-x^d_i)^4/4!=\log(\rho)(Y^d_i)-\sum_{j=0}^{3}(\log{\rho})^{(j)}(x^d_i)(Y^d_i-x^d_i)^j/j!$
for any $2\leq i\leq d$.  
Recall $Y^d_i-x^d_i$ is normal $N(0,l^2/d)$, for some constant $l>0$,
and
$\log(\rho)\in\mathcal{S}^4$. 
Hence we may apply 
Proposition~\ref{prop:BoundingSums} to the function 
$x\mapsto \log(\rho)'''(x)-2(\log(\rho)'(x))^31_{A}(x)$ 
to get
$T^d_1(\mathbf{x}^d)\leq C_1 \left(l^6/d^3\sum_{i=2}^d e^{|x^d_i|}\right)^{1/2}$ 
for some constant 
$C_1>0$, independent of $\mathbf{x}^d$.
Since $\mathbf{x}^d\in \mathcal{A}_d$, 
the assumption in~\eqref{eq:AnCond1} yields  
$T^d_1(\mathbf{x}^d)\leq C_1 l^3 (2\rho(e^{|x|}))^{1/2}/d$. 
Similarly, we apply Proposition~\ref{prop:TaylorEstimate} 
(with $f=\log\rho$, $n=k=4$, $m=1$, $s=1$ and $\sigma^2=l^2/d$) and assumption~\eqref{eq:AnCond1} 
to get $T^d_2(\mathbf{x}^d)\leq C_2 d^{-2}\sum_{i=2}^de^{|x^d_i|}\leq C_2 d^{-1}$ for some constant
$C_2>0$ and all $\mathbf{x}^d\in \mathcal{A}_d$.

Recall $f\in\mathcal{S}^3$ and let
$\tilde{W}^d_1$ be as in Proposition~\ref{prop:TaylorEstimate}, satisfying
$f'(\tilde{W}^d_1)(Y^d_1-x^d_1)=f(Y^d_1)-f(x^d_1)$.
Let $C>0$ be such that $T^d_1(\mathbf{x}^d)+T^d_2(\mathbf{x}^d)\leq Cd^{-1}$ for all $\mathbf{x}^d\in \mathcal{A}_d$.
The bound in~\eqref{eq:T_1_and_T_2_estimate},
Taylor's theorem applied to $f$ and Cauchy's inequality yield:  
\begin{eqnarray*}
&&\left|\mathcal{G}_df(\mathbf{x}^d)-\tilde{\mathcal{G}}_df(\mathbf{x}^d)\right|\leq d \mathbb{E}_{Y^d_1}\left[\left|f(Y^d_1)-f(x^d_1)\right|
\left(|\log(\rho)''(W^d_1)|(Y^d_1-x^d_1)^2+Cd^{-1}\right)\right]\nonumber\\
&&=d\mathbb{E}_{Y^d_1}\left[\left|f'(\tilde{W}^d_1)\log(\rho)''(W^d_1)(Y^d_1-x^d_1)^3\right|\right]+C\mathbb{E}_{Y^d_1}\left[\left|f'(\tilde{W}^d_1)(Y^d_1-x^d_1)\right|\right]\nonumber\\
&&\leq d\left(\mathbb{E}_{Y^d_1}\left[\left|f'(\tilde{W}^d_1)^2(Y^d_1-x^d_1)^3\right|\right]\mathbb{E}_{Y^d_1}\left[\left|\left(\log(\rho)''\right)^2(W^d_1)(Y^d_1-x^d_1)^3\right|\right]\right)^{1/2}\nonumber\\
&&\quad +C\mathbb{E}_{Y^d_1}\left[\left|f'(\tilde{W}^d_1)(Y^d_1-x^d_1)\right|\right]\leq 
\bar C d(\|f'\|_{\infty,1/2}^2 e^{|x_1^d|} d^{-3/2} \cdot  e^{|x_1^d|}d^{-3/2})^{1/2} + \bar C \|f'\|_{\infty,1/2}e^{|x_1^d|} d^{-1/2}.
\nonumber
\end{eqnarray*}
The last inequality follows by three applications of 
Proposition~\ref{prop:TaylorEstimate}, where $\bar C>0$ is a constant that does not depend on 
$f$ 
or
$\mathbf{x}^d\in \mathcal{A}_d$.  
This concludes the proof of the proposition. 
\end{proof}

Before tackling the proof of
Proposition~\ref{prop:TildaGHatG}, we need the following three lemmas.
Recall that 
$K(x,Y)$ is defined in~\eqref{eq:Mdef}
and the set $A$ satisfies the conclusion of Proposition~\ref{prop:auxAndef}.

\begin{lem}\label{lemma:Mpdf}
Pick $x\in A$ and let $Y\sim N(x,l^2/d)$ for some constant $l>0$. Then $K(x,Y)$
has a density $q_x$ satisfying $\|q_x\|_\infty\leq 4c_A\sqrt{d}/(3l\sqrt{2\pi})$.  
\end{lem}

\begin{proof} 
Existence of $q_x$ follows from~\eqref{eq:Mdef} and Proposition~\ref{prop:p(N)pdf}.
By Proposition~\ref{prop:auxAndef} we have
$\left|\log(\rho)''(x)\right|<\left(\log(\rho)'(x)\right)^2$ and
$c_A>\left|\log(\rho)'(x)\right|>1/c_A$. Consider the 
polynomial
$y\mapsto p(y):=\log(\rho)'(x)y+\log(\rho)''(x)y^2/2+\left(\log(\rho)'(x)\right)^3y^3/3$.
By~\eqref{eq:Mdef} it holds $p(Y-x)=K(x,Y)$.
Since 
$p'(y)=\log(\rho)''(x)y+\log(\rho)'(x)(1+\log(\rho)'(x)^2y^2)$,
we have
\begin{eqnarray}
|p'(y)|
&\geq&|\log(\rho)'(x)|(1+\log(\rho)'(x)^2y^2) -\left|\log(\rho)''(x)\right||y|\nonumber\\
&>&|\log(\rho)'(x)|(1-|\log(\rho)'(x)y|+|\log(\rho)'(x)y|^2)> \frac{3}{4c_A}\text{,}\nonumber
\end{eqnarray}
where 
the second inequality holds since 
$\left|\log(\rho)''(x)\right|<\left(\log(\rho)'(x)\right)^2$
and the third follows from  $\inf_{z\in\RR}\{1-|z|+z^2\}=3/4$
and $\left|\log(\rho)'(x)\right|>1/c_A$.
The lemma now follows by Proposition~\ref{prop:PolynomialPDFbound}.  
\end{proof}

Recall that the proposal is normal $\mathbf{Y}^d=(Y_1^d,\ldots,Y_d^d)\sim N(\mathbf{x}^d,l^2/d\cdot I_d)$.

\begin{lem}\label{lemma:sumMpdf}
For any $\mathbf{x}^d\in \mathcal{A}_d$, the sum 
$\sum_{k=2}^dK(x_i^d,Y_i^d)$ possesses a density $\mathbf{q}^d_{\mathbf{x}^d}$. 
Moreover, there exists a constant $C_{K}$ such that $\|\mathbf{q}^d_{\mathbf{x}^d}\|_\infty\leq
C_{K}$ holds for all $d\in\NN$ and all $\mathbf{x}^d\in \mathcal{A}_d$.
\end{lem}

\begin{proof}
Fix $\mathbf{x}^d\in\mathcal{A}_d$ and, for each $i$, let $q_i$ denote the density of
$K(x^d_i,Y^d_i)$ as in the previous lemma. 
Since the components of 
$\mathbf{Y}^d$
are IID, 
we have $\mathbf{q}^d_{\mathbf{x}^d}=\Asterisk_{i=2}^d q_i=q_A \Asterisk q_{\RR\setminus A}$,
where
$q_A:= \Asterisk_{x^d_i\in A}q_i$
and
$q_{\RR\setminus A}:=\Asterisk_{x^d_i\notin A}q_i$.
By the definition of convolution and the fact that 
$q_{\RR\setminus A}$ is a density, 
it follows that 
$\|\mathbf{q}^d_{\mathbf{x}^d}\|_\infty\leq \|q_A\|_\infty\|q_{\RR\setminus A}\|_1=\|q_A\|_\infty$.
By Lemma~\ref{lemma:Mpdf} 
there exists $C>0$ such that,   
for any $d\in\NN$,
it holds 
$\|q_i\|_\infty<C\sqrt{d}$ if $x^d_i\in A$.
Condition~\eqref{eq:AnCond2} implies there are at least 
$(d-1)\rho(A)/2$ 
factors in the convolution $q_A=\Asterisk_{x^d_i\in A}q_i$. Hence
Proposition~\ref{prop:SumPolynomialPDFboun} applied to $q_A$ yields
$\|\mathbf{q}^d_{\mathbf{x}^d}\|_\infty\leq\|q_A\|_\infty\leq c  \frac{C\sqrt{d}}{\sqrt{(d-1)\rho(A)/2}}$.
This concludes the proof of the lemma. 
\end{proof}

\begin{lem}\label{lemma:betaProperties}
Let $\mathbf{x}^d\in \mathcal{A}_d$. The function $y\mapsto\beta(\mathbf{x}^d,y)$, defined in~\eqref{eq:betadef}, 
is in $\mathcal{C}^2(\RR)$ and the following holds:  
\begin{enumerate}[(i)]
\item $0<\beta(\mathbf{x}^d,y)\leq1$ for all $y\in\RR$;
\item $\beta(\mathbf{x}^d,x^d_1)=\mathbb{E}_{\mathbf{Y}^d-}\left[1\wedge e^{\sum_{i=2}^dK(x^d_i,Y^d_i)}\right]$;
\item $\left|\frac{\partial}{\partial y}\beta(\mathbf{x}^d,y)\right|\leq \left|\log(\rho)'(x^d_1)\right|$ for all $y\in\RR$;
\item $\frac{\partial}{\partial y}\beta(\mathbf{x}^d,x^d_1)=\log(\rho)'(x^d_1)\mathbb{E}_{\mathbf{Y}^d-}\left[e^{\sum_{i=2}^dK(x^d_i,Y^d_i)}1_{\left\{\sum_{i=2}^dK(x^d_i,Y^d_i)<0\right\}}\right]$;
\item $\left|\frac{\partial^2}{\partial y^2}\beta(\mathbf{x}^d,y)\right|\leq \left|\log(\rho)'(x^d_1)\right|^2(C_K+1)$ for all $y\in\RR$ and constant $C_K$ from Lemma~\ref{lemma:sumMpdf}.
\end{enumerate}
\end{lem}

\begin{proof} (i) and (ii) follow from the definition in~\eqref{eq:betadef}. 
Since $x\mapsto 1\wedge e^x$ is Lipschitz (with Lipschitz constant $1$) on $\RR$, 
the family of functions $\{x\mapsto (1\wedge e^{x+h}- 1\wedge e^x)/h; h\in\RR\setminus\{0\}\}$
is bounded by one and converges pointwise to $1_{\{x<0\}}e^x$ for all $x\in\RR\setminus\{0\}$,
as $h\to0$.
Hence the DCT implies that 
$\frac{\partial}{\partial y}\beta(\mathbf{x}^d,y)$ 
exists and can be expressed as 
\begin{eqnarray}
\label{eq:beta_derivative}
&\log(\rho)'(x^d_1)\mathbb{E}_{\mathbf{Y}^d-}\left[e^{\log(\rho)'(x^d_1)(y-x^d_1)
+\sum_{i=2}^dK(x^d_i,Y^d_i)}1_{\left\{\log(\rho)'(x^d_1)(y-x^d_1)+\sum_{i=2}^dK(x^d_i,Y^d_i)<0\right\}}\right],&
\end{eqnarray}
implying (iii) and (iv). 
Let $\Phi^d_{K}$ denote the distribution of $\sum_{i=2}^dK(x^d_i,Y^d_i)$
and recall that by definition we have $e^x 1_{\{x<0\}} = 1\wedge e^x - 1_{\{x\geq 0\}}$ for all $x\in\RR$.
Hence, by~\eqref{eq:beta_derivative}, it follows 
\begin{eqnarray}
\frac{\partial}{\partial y}\beta(\mathbf{x}^d,y) & = &
\log(\rho)'(x^d_1)\left(\beta(\mathbf{x}^d,y)-1+\Phi^d_{K}\left(-\log(\rho)'(x^d_1)(y-x^d_1)\right)\right)\text{,}
\label{eq:beta_derivative_1}
\end{eqnarray}
By Lemma~\ref{lemma:sumMpdf}, $\Phi_K^d$ is differentiable. Hence, by~\eqref{eq:beta_derivative_1}, 
$\frac{\partial^2}{\partial y^2}\beta(\mathbf{x}^d,y)$ also exists and takes the form: 
$$\left(\log(\rho)'(x^d_1)\right)^2\left(\beta(\mathbf{x}^d,y)-1
+\Phi^d_{K}\left(-\log(\rho)'(x^d_1)(y-x^d_1)\right)-\mathbf{q}^d_{\mathbf{x}^d}\left(-\log(\rho)'(x^d_1)(y-x^d_1)\right)\right)\text{.}$$
Part (v) follows from this representation of $\frac{\partial^2}{\partial y^2}\beta(\mathbf{x}^d,y)$
and Lemma~\ref{lemma:sumMpdf}.
\end{proof}

\begin{proof}[Proof of Proposition~\ref{prop:TildaGHatG}]
Fix an arbitrary $\mathbf{x}^d\in\mathcal{A}_d$. 
Let 
$Z_1,W_1$ be random variables, as in
Proposition~\ref{prop:TaylorEstimate},
that satisfy
\begin{eqnarray*}
f(Y^d_1)-f(x^d_1) & = & f'(x^d_1)(Y^d_1-x^d_1)+\frac{f''(x^d_1)}{2}(Y^d_1-x^d_1)^2+\frac{f'''(Z_1)}{6}(Y^d_1-x^d_1)^3, \\
\beta(\mathbf{x}^d,Y^d_1) & = & \beta(\mathbf{x}^d,x^d_1)+\frac{\partial}{\partial
y}\beta(\mathbf{x}^d,x^d_1)(Y^d_1-x^d_1)+\frac{\frac{\partial^2}{\partial
y^2}\beta(\mathbf{x}^d,W_1)}{2}(Y^d_1-x^d_1)^2.
\end{eqnarray*}
Then, by the definition of 
$\tilde{\mathcal{G}}_df(\mathbf{x}^d)$
in~\eqref{eq:GtildaGenerator} and the fact $Y_1^d-x_1^d\sim N(0,l^2/d)$, we find 
\begin{eqnarray}
\tilde{\mathcal{G}}_df(\mathbf{x}^d)&=&\frac{l^2f''(x^d_1)}{2}\beta(\mathbf{x}^d,x^d_1)+l^2f'(x^d_j)\frac{\partial}{\partial y}\beta(\mathbf{x}^d,x^d_1)\nonumber\\
&+&d\mathbb{E}_{Y^d_1}\left[\left(\beta(\mathbf{x}^d,x^d_1)\frac{f'''(Z_1)}{6}+f'(x^d_1)\frac{\frac{\partial^2}{\partial y^2}\beta(\mathbf{x}^d,W_1)}{2}\right)(Y^d_1-x^d_1)^3\right]\nonumber\\
&+&d\mathbb{E}_{Y^d_1}\left[\left(\frac{f''(x^d_1)}{2}\frac{\frac{\partial^2}{\partial y^2}\beta(\mathbf{x}^d,W_1)}{2}+\frac{f'''(Z_1)}{6}\frac{\partial}{\partial y}\beta(\mathbf{x}^d,x^d_1)\right)(Y^d_1-x^d_1)^4\right]\nonumber\\
&+&d\mathbb{E}_{Y^d_1}\left[\frac{f'''(Z_1)}{6}\frac{\frac{\partial^2}{\partial y^2}\beta(\mathbf{x}^d,W_1)}{2}(Y^d_1-x^d_1)^5\right]\text{.}\nonumber
\end{eqnarray}
By parts (ii) and (iv) in Lemma~\ref{lemma:betaProperties}
and the definition of $\hat{\mathcal{G}}_df(\mathbf{x}^d)$ in~\eqref{eq:GhatGenerator}
we have
$\hat{\mathcal{G}}_df(\mathbf{x}^d)=
\frac{l^2f''(x^d_1)}{2}\beta(\mathbf{x}^d,x^d_1)+l^2f'(x^d_j)\frac{\partial}{\partial y}\beta(\mathbf{x}^d,x^d_1)$. 
The three expectations in the display above can each be bounded
by a constant times $\left(\sum_{i=1}^3\|f^{(i)}\|_{\infty,1/2}\right)e^{|x^d_1|}d^{-1/2}$ using
Proposition~\ref{prop:TaylorEstimate}~and Lemma~\ref{lemma:betaProperties}.
For instance, the first expectation can be bounded above using 
(v) in Lemma~\ref{lemma:betaProperties}: 
$$\frac{d}{6}\mathbb{E}_{Y^d_1}\left[|f'''(Z_1)||Y^d_1-x^d_1|^3\right]+\frac{d(C_K+1)}{2}|f'(x^d_1)||\log(\rho'(x^d_1))|^2\mathbb{E}_{Y^d_1}\left[|Y^d_1-x^d_1|^3\right]\text{.}$$
Proposition~\ref{prop:TaylorEstimate} yields 
$\frac{d}{6}\mathbb{E}_{Y^d_1}\left[|f'''(Z_1)||Y^d_1-x^d_1|^3\right]
\leq C_0 e^{|x^d_1|} \|f'''\|_{\infty,1} d^{-1/2}\leq C_0 e^{|x^d_1|} \|f'''\|_{\infty,1/2} d^{-1/2}$ for some $C_0>0$. 
Moreover, $|f'(x^d_1)||\log(\rho'(x^d_1))|^2\leq \|f'\|_{\infty,1/2}\|(\log(\rho)')^2\|_{\infty,1/2}e^{|x^d_1|}$
as $\log(\rho)\in\mathcal{S}^4$ and $f\in\mathcal{S}^3$.
Hence
$\frac{d(C_K+1)}{2}|f'(x^d_1)||\log(\rho'(x^d_1))|^2\mathbb{E}_{Y^d_1}\left[|Y^d_1-x^d_1|^3\right]
\leq C_1 e^{|x^d_1|} \|f'\|_{\infty,1/2} d^{-1/2}$
for some $C_1>0$. 
Similarly, it follows that the second and third expectations above decay as $d^{-1}$ and $d^{-3/2}$, respectively.
This concludes the proof of the proposition. 
\end{proof}

\begin{lem}\label{lemma:MNdifference}
Recall that $\mathcal{N}(\mathbf{x}^d,\mathbf{Y}^d)$ and $\sum_{i=2}^dK(x^d_i,Y^d_i)$
are defined in~\eqref{eq:Ndef} and \eqref{eq:Mdef}, respectively. Then there
exists a constant $C$ such that for all $d\in\NN$ we have:
$$\mathbb{E}_{\mathbf{Y}^d-}\left[\left|\sum_{i=2}^dK(x^d_i,Y^d_i)-\mathcal{N}(\mathbf{x}^d,\mathbf{Y}^d)\right|\right]\leq
Cd^{-1/2}\quad \forall\mathbf{x}^d\in\mathcal{A}_d\text{.}$$ 
\end{lem}

\begin{proof}
The difference in question is smaller than the sum of the following two terms:
\begin{eqnarray}
T^d_3(\mathbf{x}^d)&=&\mathbb{E}_{\mathbf{Y}^d-}\left[\left|\sum_{i=2}^d\frac{\log(\rho)''(x^d_i)}{2}\left((Y^d_i-x^d_i)^2-\frac{l^2}{d}\right)\right|\right]\text{,}\nonumber\\
T^d_4(\mathbf{x}^d)&=&\mathbb{E}_{\mathbf{Y}^d-}\left[\left|\sum_{i=2}^d\frac{\left(\log(\rho)'(x^d_i)\right)^31_{A}(x^d_i)}{3}(Y^d_i-x^d_i)^3\right|\right]\text{.}\nonumber
\end{eqnarray}
Note that $X_i:=(Y^d_i-x^d_i)^2-l^2/d$, $2\leq i\leq d$, are zero mean IID with $\mathbb{E}[X_i^4] = 2l^4/d^2$.
Hence, as $\log(\rho)\in\mathcal{S}^4$, we may apply
Proposition~\ref{prop:BoundingSums} with the function $x\mapsto \log(\rho)''(x)$ 
and $X_i$,$2\leq i\leq d$,  
to get 
$T^d_3(\mathbf{x}^d)\leq  \|\log(\rho)''\|_{\infty,1/2}\left(2(l^4/d^2)\sum_{i=2}^d e^{|x^d_i|}\right)^{1/2}\leq C_0 d^{-1/2}$ for some
constant
$C_0>0$,
where the second inequality follows from~\eqref{eq:AnCond1}. 
Similarly,
Proposition~\ref{prop:BoundingSums} and assumption~\eqref{eq:AnCond1}, applied to 
the function $x\mapsto (\log(\rho)'(x))^31_{A}(x)$ and random variables
$Y^d_i-x^d_i$, yield $T^d_4(\mathbf{x}^d)\leq C_1 \left(\sum_{i=2}^d
e^{|x^d_i|}/d^3\right)^{1/2}\leq C_2 d^{-1}$ for some constants $C_1,C_2>0$ and all $d\in\NN$. 
\end{proof}

\begin{lem}\label{lemma:Acharacteristic}
There exist constants $c_1,c'_1>0$, such that
for any $d\in\NN$, $i\in\{2,\cdots,d\}$, $\mathbf{x}^d\in\mathcal{A}_d$ and  $x^d_i\notin A$,
it holds 
$$\left|\log\varphi_i(t)-\left(i\frac{l^2}{2d}\log(\rho)''(x^d_i)t-\frac{l^2}{2d}\left(\log(\rho)'(x^d_i)\right)^2t^2\right)\right|\leq
\frac{c_1}{d^{3/2}}(t^2+|t|^3)\text{,}\quad |t|\leq c'_1\sqrt{d},$$
where
$\varphi_i(t):=\mathbb{E}_{Y_i^d}[\exp(itK(x^d_i,Y^d_i))]$, $t\in\RR$, is the CF of $K(x^d_i,Y^d_i)$
(cf.~\eqref{eq:Mdef}).
\end{lem}

\begin{rema}
Recall that the set $A$ satisfies the conclusion of Proposition~\ref{prop:auxAndef}. 
The proof of 
Lemma~\ref{lemma:Acharacteristic} requires the control of the functions $\log(\rho)'$ and $\log(\rho)''$
on the complement of $A$, where they are unbounded. It is hence crucial that their argument $x_i^d$ is 
the $i$-th coordinate of a point $\mathbf{x}^d\in\mathcal{A}_d$, 
since, through assumption~\eqref{eq:AnCond1},
we have control over the size of $x_i^d$ in terms of the dimension $d$ of the chain. 
For an analogous reason we need $i>1$. These facts plays a key role in the proof. 
\end{rema}

\begin{proof}
By Lemma~\ref{lemma:CharaN+bN^2},
the following inequality holds for all $|t|\leq d/(4l^2|\log(\rho)''(x^d_i)|)$:
\begin{multline}
\left|\log\varphi_i(t)-\left(i\frac{l^2}{2d}\log(\rho)''(x^d_i)t-\frac{l^2}{2d}\left(\log(\rho)'(x^d_i)\right)^2t^2\right)\right|\\
\leq
\frac{l^4}{d^2}\left(\frac{1}{2}\left(\log(\rho)''(x^d_i)\right)^2t^2+\left(\log(\rho)'(x^d_i)\right)^2\left|\log(\rho)''(x^d_i)\right||t|^3\right)\text{,}\quad
\label{eq:Ineq_math_cal_A_d}
\end{multline}
Since $\mathbf{x}^d\in\mathcal{A}_d$ and $i>1$,
assumption~\eqref{eq:AnCond1}
implies that 
for any 
$f\in\mathcal{S}^0$ there exists $C_f>0$ such that 
$|f(x^d_i)|^2/C_f\leq e^{|x^d_i|}\leq \sum_{i=2}^de^{|x^d_i|}\leq 2d\rho(e^{|x|})$. 
Hence
$|f(x^d_i)|\leq c_f \sqrt{d}$ for all $d\in\NN$ and all $2\leq i\leq d$, where $c_f:=(2C_f\rho(e^{|x|}))^{1/2}$.
Since $\log(\rho)\in\mathcal{S}^4$, both functions $f_1(x):=l^4(\log(\rho)''(x))^2/2$ 
and 
$f_2(x):=(\log(\rho)'(x))^2|\log(\rho)''(x)|$ 
are in $\mathcal{S}^0$. 
Then~\eqref{eq:Ineq_math_cal_A_d} and the constants $c_1:=\max\{c_{f_1},c_{f_2}\}$ 
and $c_1':=1/(4\sqrt{2c_{f_1}})$ yield the inequalities in the lemma. 
\end{proof}

We now deal with the coordinates of $\mathcal{A}_d$ that are in $A$. Compared to 
Lemma~\ref{lemma:Acharacteristic},
this is straightforward as it does not involve the remainder of the coordinates 
of the point in $\mathcal{A}_d$. 

\begin{lem}\label{lemma:BerryAux}
If $x\in A$, then $K(x,Y)$ (cf.~\eqref{eq:Mdef}), where $Y\sim N(x,l^2/d)$, satisfies: 
\begin{enumerate}[(a)]
\item $\mu_K:=\mathbb{E}_{Y}\left[K(x,Y)\right]\leq \frac{l^2c_A^2}{2d}$, where $c_A>0$ is the constant in Proposition~\ref{prop:auxAndef};
\item $|\mathbb{E}_{Y}[(K(x,Y)-\mu_K)^2]- (\log(\rho)'(x))^2\frac{l^2}{d}| \leq C_1 d^{-2}$ for some 
constant $C_1>0$ and all $d\in\NN$;
\item
$\mathbb{E}_{Y}\left[\left|K(x,Y)-\mu_K\right|^3\right]\leq C_2d^{-3/2}$ for some
constant $C_2>0$ and all $d\in\NN$.
\end{enumerate}
Moreover, the constants $C_1$ and $C_2$ do not depend on the choice of $x\in A$.
\end{lem}

\begin{proof}
By definition of $A$ in Proposition~\ref{prop:auxAndef} we have
$\left|\log(\rho)'(x)\right|\leq c_A$ and $\left|\log(\rho)''(x)\right|\leq
c^2_A$ for $x\in A$. 
By~\eqref{eq:Mdef},
$\mu_K=\frac{l^2}{2d}\log(\rho)''(x)$ and (a) follows.
Recall $\mathbb{E}_{Y}[(Y-x)^n]$ is either zero (if $n$ is odd) or of order $d^{-n/2}$ (if $n$ is even)
and $\mathbb{E}_{Y}[(Y-x)^2]=l^2/d$. Hence the definition of $K$ in~\eqref{eq:Mdef}, the fact $x\in A$ and part~(a) imply the inequality in 
part~(b). For part~(c), note that an analogous argument yields
$\mathbb{E}_{Y}\left[(K(x,Y)-\mu_K)^6\right]\leq C'd^{-3}$ for some constant $C'>0$. Cauchy's inequality concludes the proof of the lemma. 
\end{proof}

\begin{lem}\label{lemma:nonAcharacteristic}
Let assumptions of Lemma~\ref{lemma:BerryAux} hold 
and denote by $\varphi$ the characteristic function of $K(x,Y)$.  
There exist positive constants $c_2$ and $c'_2$, such that the following holds for all $x\in A$:
\begin{equation}
\label{eq:nonAcharacteristic}
\left|\log\varphi(t)-\left(it\frac{l^2}{2d}\log(\rho)''(x)-\frac{t^2l^2}{2d}\log(\rho)'(x)^2\right)\right|\leq
c_2\left(\frac{t^2}{d^2}+\frac{|t|^3}{d^{3/2}}+\frac{t^4}{d^2}\right)\text{,}\quad
|t|\leq c'_2\sqrt{d}\text{.}
\end{equation}
\end{lem}

\begin{proof}
Let $\sigma^2_K:=\mathbb{E}_{Y}[(K(x,Y)-\mu_K)^2]$ and recall $\mu_K=\frac{l^2}{2d}\log(\rho)''(x)$. 
By Lemma~\ref{prop:NearNormalChar} we have
\begin{equation}\label{eq:auxAchar}
\left|\log\varphi(t)-\left(it\frac{l^2}{2d}\log(\rho)''(x)-\frac{t^2}{2}\sigma^2_K\right)\right|
\leq|t|^3\mathbb{E}_{Y}\left|K(x,Y)-\mu_K\right|^3/6+t^4\sigma^4_K/4\text{,}\quad |t|\leq \frac{1}{\sigma_K}\text{.}
\end{equation}
By  Lemma~\ref{lemma:BerryAux}(b) we  have 
$|\sigma^2_K-l^2\log(\rho)'(x)^2/d|\leq C_1d^{-2}$.
Hence 
$\sigma^2_K\leq d^{-1}/\sqrt{c_2'}$,
where
$c_2':= 1/(l^2c_A^2+ C_1)^2$,
and
$\sigma^4_K\leq C_1'd^{-2}$ for some $C_1'>0$.
This, together with Lemma~\ref{lemma:BerryAux}(c), implies that there exists  a constant $c_2>0$, 
such that the inequality in~\eqref{eq:nonAcharacteristic} follows from~\eqref{eq:auxAchar}
for all $|t|\leq c_2' d^{1/2} \leq 1/\sigma_K$ and $x\in A$. 
\end{proof}

\begin{lem}\label{lemma:CFtoCDF}
For any $d\in\NN$ and $\mathbf{x}^d\in \mathcal{A}_d$, let  $\Phi^d_{K}$ and
$\Phi^d_{\mathcal{N}}$ be the distribution functions of $\sum_{i=2}^dK(x^d_i,Y^d_i)$ and
$\mathcal{N}(\mathbf{x}^d,\mathbf{Y}^d)$. Then there exists $C>0$ such
that   
$$\sup_{x\in\RR}\left|\Phi^d_{\mathcal{N}}(x)-\Phi^d_{K}(x)\right|\leq C
d^{-1/2}\qquad \text{for every $d\in\NN$ and $\mathbf{x}^d\in \mathcal{A}_d$.}$$
\end{lem}

\begin{proof}
Let $\varphi_K$ and $\varphi_\mathcal{N}$  be the CFs of
$\sum_{i=2}^dK(x^d_i,Y^d_i)$ and $\mathcal{N}(\mathbf{x}^d,\mathbf{Y}^d)$, respectively. 
We will compare $\varphi_K$ and $\varphi_\mathcal{N}$   
and apply Proposition~\ref{prop:Berry} to establish the lemma. 
Let $\varphi_i$ be the CF of $K(x^d_i,Y^d_i)$ and recall, by~\eqref{eq:Ndef}, 
$\varphi_\mathcal{N}(t)=\exp\left(\frac{1}{2}it\frac{l^2}{d}\sum_{i=2}^d\log(\rho)''(x^d_i)-\frac{1}{2}t^2\frac{l^2}{d}\sum_{i=2}^d\left(\log(\rho)'(x^d_i)\right)^2\right)$.
Define the positive constants
$c:=\max\{c_1,c_2\}$ and $c':=\min\{1,l^2J/(32c),c'_1,c'_2\}$, where 
the constants $c_1,c_1'$ (resp. $c_2,c_2'$) are given in Lemma~\ref{lemma:Acharacteristic} 
(resp. Lemma~\ref{lemma:nonAcharacteristic}) and $J$ is as in assumption~\eqref{eq:AnCond3}. 
Note that the constants $c,c'$ do not depend on the choice of $\mathbf{x}^d\in \mathcal{A}_d$.
Lemmas~\ref{lemma:Acharacteristic} and~\ref{lemma:nonAcharacteristic} imply the following inequality 
for all $d\in\NN$ and 
$\mathbf{x}^d\in \mathcal{A}_d$:
\begin{equation*}\left|\log\varphi_K(t)-\log\varphi_\mathcal{N}(t)\right|\leq
\sum_{i=2}^d\left|\log\varphi_i(t)-\left(it\frac{l^2}{2d}\log(\rho)''(x^d_i)-\frac{t^2l^2}{2d}\left(\log(\rho)(x^d_i\right)^2\right)\right| \leq
R(t), 
\end{equation*}
for all $|t|\leq r$, where $r:= c'\sqrt{d}$
and
$R(t):= c(t^2+|t|^3+t^4/\sqrt{d})/\sqrt{d}$.
Since $|t|^3\leq \sqrt{d}c't^2$ and $t^4\leq dc'^2t^2$ for 
$|t|\leq r$,
we have 
\begin{equation}
\label{eq:auxBerry}
 R(t)\leq t^2(c/\sqrt{d}+cc'+cc'^2) \leq t^2 (c/\sqrt{d}+2cc')\qquad \text{for all $t\in[-r,r]$.}
\end{equation}

By assumption~\eqref{eq:AnCond3}, there exists $d_0'\in\NN$ such that 
the variance 
$\sigma^2_{\mathcal{N}}(\mathbf{x}^d)=l^2/d\sum_{i=2}^d\left(\log(\rho)'(x^d_i)\right)^2$
of 
$\mathcal{N}(\mathbf{x}^d,\mathbf{Y}^d)$ 
satisfies $\sigma^2_{\mathcal{N}}(\mathbf{x}^d)\geq l^2J/2$ for all 
$d\geq d_0'$ and  $\mathbf{x}^d\in \mathcal{A}_d$. 
Let $\gamma:=1/2$ and pick 
$d_0\in\NN$, greater than 
$\max\{d_0',(16cl^{-2}/J)^2\}$.
Then, for any $d\geq d_0$, the inequality $c/\sqrt{d}\leq \gamma l^2J/8$ holds.
Since $c'\leq l^2J/(32c)$, we have 
$2cc'\leq \gamma l^2J/8$, 
and the bound in~\eqref{eq:auxBerry} implies
$R(t)\leq 
\frac{1}{2}t^2 \gamma l^2J/2
\leq 
\frac{1}{2}t^2 \gamma \sigma^2_{\mathcal{N}}(\mathbf{x}^d)$
for all $t\in[-r,r]$.
By Proposition~\ref{prop:Berry}, 
for all $d\geq d_0$,
$\sup_{x\in\RR}\left|\Phi^d_{\mathcal{N}}(x)-\Phi^d_{K}(x)\right|$ is bounded above 
by
\begin{equation*}
\int_{\RR}\frac{R(t)}{\pi |t|}\exp\left(-\frac{(1-\gamma)\sigma^2_{\mathcal{N}}(\mathbf{x}^d)t^2}{2}\right)dt
+\frac{12\sqrt{2}}{ \pi^{3/2}\sigma_{\mathcal{N}}(\mathbf{x}^d) r}\leq C'/\sqrt{d},
\end{equation*}
where
$C':= c\int_\RR(|t|+t^2+|t|^3)\exp(-l^2Jt^2/8)dt
+ \frac{24\sqrt{2}}{\pi^{3/2}l^2Jc'}$.
Since the left-hand side of the inequality in the lemma is bounded above by $1$,
the inequality holds for all $d\in\NN$ if we define 
$C:=\max\{C',\sqrt{d_0}\}$. 
\end{proof}

\begin{proof}[Proof of Proposition~\ref{prop:HatGBreveG}]
Since $|1\wedge e^y- 1\wedge e^x|\leq |x-y|$ for all $x,y\in\RR$, by Lemma~\ref{lemma:MNdifference} we have 
$$\left|\mathbb{E}_{\mathbf{Y}^d-}\left[1\wedge
e^{\mathcal{N}(\mathbf{x}^d,\mathbf{Y}^d)}\right]-\mathbb{E}_{\mathbf{Y}^d-}\left[1\wedge
e^{\sum_{i=2}^dK(x^d_i,Y^d_i)}\right]\right|\leq C' d^{-1/2}$$
for some constant $C'>0$ and all $d\in\NN$.
Recall
$e^x 1_{\{x<0\}} = 1\wedge e^x - 1 + 1_{\{x\leq 0\}}$ for all $x\in\RR\setminus\{0\}$.
Hence Lemmas~\ref{lemma:MNdifference} and~\ref{lemma:CFtoCDF} yield
\begin{multline*}
\left|\mathbb{E}_{\mathbf{Y}^d-}\left[e^{\mathcal{N}(\mathbf{x}^d,\mathbf{Y}^d)}1_{\left\{\mathcal{N}(\mathbf{x}^d,\mathbf{Y}^d)<0\right\}}\right]-\mathbb{E}_{\mathbf{Y}^d-}\left[e^{\sum_{i=2}^dK(x^d_i,Y^d_i)}1_{\left\{\sum_{i=2}^dK(x^d_i,Y^d_i)<0\right\}}\right]\right|\\
\leq \left|\mathbb{E}_{\mathbf{Y}^d-}\left[1\wedge
e^{\mathcal{N}(\mathbf{x}^d,\mathbf{Y}^d)}\right]-\mathbb{E}_{\mathbf{Y}^d-}\left[1\wedge
e^{\sum_{i=2}^dK(x^d_i,Y^d_i)}\right]\right|+\left|\Phi^d_{\mathcal{N}}(0)-\Phi^d_{K}(0)\right|\leq
C'' d^{-1/2}
\end{multline*}
for some $C''>0$ and all $d\in \NN$. The proposition follows. 
\end{proof}

\begin{proof}[Proof of Proposition~\ref{prop:BreveGG}]
For any $\mathbf{x}^d\in\mathcal{A}_d$, by \cite[Proposition~2.4]{RobertsGelmanGilks97}, we have
\begin{eqnarray}
&&\mathbb{E}_{\mathbf{Y}^d-}\left[e^{\mathcal{N}(\mathbf{x}^d,\mathbf{Y}^d)}1_{\left\{\mathcal{N}(\mathbf{x}^d,\mathbf{Y}^d)<0\right\}}\right]=e^{\mu_{\mathcal{N}}(\mathbf{x}^d)+\frac{\sigma^2_{\mathcal{N}}(\mathbf{x}^d)}{2}}\Phi\left(-\sigma_{\mathcal{N}}(\mathbf{x}^d)-\frac{\mu_{\mathcal{N}}(\mathbf{x}^d)}{\sigma_{\mathcal{N}}(\mathbf{x}^d)}\right)\text{,}\nonumber\\
&&\mathbb{E}_{\mathbf{Y}^d-}\left[1\wedge e^{\mathcal{N}(\mathbf{x}^d,\mathbf{Y}^d)}\right]=
\mathbb{E}_{\mathbf{Y}^d-}\left[e^{\mathcal{N}(\mathbf{x}^d,\mathbf{Y}^d)}1_{\left\{\mathcal{N}(\mathbf{x}^d,\mathbf{Y}^d)<0\right\}}\right]
+
\Phi\left(\frac{\mu_{\mathcal{N}}(\mathbf{x}^d)}{\sigma_{\mathcal{N}}(\mathbf{x}^d)}\right)\text{.}\nonumber
\end{eqnarray}
where $\Phi$ is the distribution of a standard normal random variable.
Note first that it is sufficient to prove the inequality in the proposition for all $d> d_0$ for some $d_0\in\NN$,
since the expectations above are bounded by $1$ and we can hence increase the constant $C$ so that 
the first $d_0$ inequalities are also satisfied. 

Recall the formulas for $\mu_{\mathcal{N}}(\mathbf{x}^d)$ and
$\sigma^2_{\mathcal{N}}(\mathbf{x}^d)$ from \eqref{eq:Ndef}. By assumptions~\eqref{eq:AnCond3} and~\eqref{eq:AnCond4} it follows that 
$\left|\mu_{\mathcal{N}}(\mathbf{x}^d)+\sigma^2_{\mathcal{N}}(\mathbf{x}^d)/2\right|\leq c a_d/\sqrt{d}$ for some constant $c>0$ 
and all  large $d$
and $\mathbf{x}^d\in\mathcal{A}_d$. 
Note that $S_a:=\sup_{d\in\NN}(a_d/\sqrt{d})<\infty$ since $\{a_d\}_{d\in\NN}$ is sluggish. 
The function $x\mapsto e^x$ is Lipschitz on $[-cS_a,cS_a]$ with constant $e^{cS_a}$. 
Consequently $\left|e^{\mu_{\mathcal{N}}(\mathbf{x}^d)+\sigma^2_{\mathcal{N}}(\mathbf{x}^d)/2}-1\right|\leq e^{cS_a}a_d/\sqrt{d}$ 
for large $d$ and uniformly in $\mathbf{x}^d\in\mathcal{A}_d$.

By assumption~\eqref{eq:AnCond3}, 
for all large  $d\in\NN$ and all $\mathbf{x}^d\in\mathcal{A}_d$, we have
$\sigma_{\mathcal{N}}(\mathbf{x}^d)\geq l\sqrt{J}/\sqrt{2}$. 
Hence, since the function $x\mapsto \sqrt{x}$ is Lipschitz with
constant $c_1:=1/(l\sqrt{2J})$ on $[\frac{l^2J}{2},\infty)$, we get
$\left|\sigma_{\mathcal{N}}(\mathbf{x}^d)/2-l\sqrt{J}/2\right|\leq (c_1/2) \left|\sigma^2_{\mathcal{N}}(\mathbf{x}^d)-l^2J\right|
\leq c_2 a_d/\sqrt{d}$, where constant $c_2>0$ exists by~\eqref{eq:AnCond3}. 
Moreover,
$|(\mu_{\mathcal{N}}(\mathbf{x}^d)+\sigma^2_{\mathcal{N}}(\mathbf{x}^d)/2)/\sigma_{\mathcal{N}}(\mathbf{x}^d)|\leq c_3 a_d/\sqrt{d}$ 
for $c_3>0$ and all large $d$. 

Since
$\sigma_{\mathcal{N}}(\mathbf{x}^d)+\mu_{\mathcal{N}}(\mathbf{x}^d)/\sigma_{\mathcal{N}}(\mathbf{x}^d)
=\left(\mu_{\mathcal{N}}(\mathbf{x}^d)+
\sigma^2_{\mathcal{N}}(\mathbf{x}^d)/2\right)/\sigma_{\mathcal{N}}(\mathbf{x}^d)+\sigma_{\mathcal{N}}(\mathbf{x}^d)/2$,
the inequalities in the previous paragraph imply that there exists $c_4>0$ such that 
$|\sigma_{\mathcal{N}}(\mathbf{x}^d)+\mu_{\mathcal{N}}(\mathbf{x}^d)/\sigma_{\mathcal{N}}(\mathbf{x}^d) - l\sqrt{J}/2|\leq c_4 a_d/\sqrt{d}$
for large $d$ and uniformly in $\mathbf{x}^d\in\mathcal{A}_d$.
Since $\Phi$ is Lipschitz with constant
$1/\sqrt{2\pi}$, there exists a constant $C_1'>1$, 
such that 
$$\left|\mathbb{E}_{\mathbf{Y}^d-}\left[e^{\mathcal{N}(\mathbf{x}^d,\mathbf{Y}^d)}1_{\left\{\mathcal{N}(\mathbf{x}^d,\mathbf{Y}^d)<0\right\}}\right]-\Phi\left(\frac{-l\sqrt{J}}{2}\right)\right|\leq C_1'\frac{a_d}{\sqrt{d}}
$$
holds 
for 
all large $d$ and all 
$\mathbf{x}^d\in\mathcal{A}_d$. 
Similarly, 
$\left|\mathbb{E}_{\mathbf{Y}^d-}\left[1\wedge e^{\mathcal{N}(\mathbf{x}^d,\mathbf{Y}^d)}\right]
-2\Phi\left(\frac{-l\sqrt{J}}{2}\right)\right|\leq C_2'\frac{a_d}{\sqrt{d}}$
for some $C_2'>0$
all large $d$ and all 
$\mathbf{x}^d\in\mathcal{A}_d$, 
and the proposition follows. 
\end{proof}

\subsection{Proof of Proposition~\ref{cor:Lpnorm}}
\label{subsec:proofthm2}
We will now prove the following result.

\begin{prop}\label{thm:Lp}
Let $a=\{a_d\}_{d\in\NN}$ be a sluggish sequence and $p\in[1,\infty)$. There
exists a constant $C_4$ (depending on $a$ and $p$) such that for every
$f\in\mathcal{S}^3$ and all $d\in\NN$ we have:
$$\left\|\mathcal{G}f-{\mathcal{G}}_df\right\|_p\leq
C_4\left(\sum_{i=1}^3\|f^{(i)}\|_{\infty,1/2}\right) \left(\frac{a_d}{\sqrt{d}}
+e^{-a_d^2/p}\right)\text{.}$$
\end{prop}

In the case $p=2$, define $a_d:=\sqrt{2\log(d)}$ for $d\in\NN\setminus\{1\}$
and note that Proposition~\ref{cor:Lpnorm} then 
follows as a special case of Proposition~\ref{thm:Lp}.

\begin{lem}\label{lemma:AuxLp1}
There exists a constant $C$ such that for all $f\in\mathcal{S}^3$ and all
$d\in\NN$ we have: $$\max\left\{\left|\mathcal{G}f(\mathbf{x}^d)\right|,\left|\mathcal{G}_df(\mathbf{x}^d)\right|\right\}
\leq C
e^{|x^d_1|}\sum_{i=1}^2\|f^{(i)}\|_{\infty,1/2}\quad\forall
\mathbf{x}^d\in\RR^d\text{.}$$
\end{lem}

\begin{proof}
The triangle inequality, definition~\eqref{eq:GenU} and $\log(\rho)'(x)f'(x)\leq \|\log(\rho)'\|_{\infty,1/2}\|f'\|_{\infty,1/2}e^{|x|}$
imply the bound in the lemma for $|\mathcal{G}f(\mathbf{x}^d)|$.
To bound $|\mathcal{G}_df(\mathbf{x}^d)|$,
define
$$\tilde{\beta}(\mathbf{x}^d,y):=\mathbb{E}_{\mathbf{Y}^d-}\left[1\wedge
\exp\left(\log(\rho)(y)-\log(\rho)(x^d_1)+\sum_{i=2}^d\log(\rho)(Y^d_i)-\log(\rho)(x^d_i)\right)\right]$$
for any $y\in\RR$. 
Then, if $q$ denotes the density of $Y^d_1-x^d_1\sim N(0,l^2/d)$, we get 
\begin{multline}
\left|\mathcal{G}_df(\mathbf{x}^d)\right|=d\left| \mathbb{E}_{Y^d_1}\left[\left(f(Y^d_1)-f(x^d_1)\right)\tilde{\beta}(\mathbf{x}^d,Y^d_1)\right]\right|
\\
\leq d\int_0^{\infty}z\left|f'(w_1)\tilde{\beta}(\mathbf{x}^d,x^d_1+z)-f'(w_2)\tilde{\beta}(\mathbf{x}^d,x^d_1-z)\right|q(z)dz\text{,}\label{eq:auxGdBound}
\end{multline}
where $w_1\in(x_1^d,x_1^d+z)$ and $w_2\in(x_1^d-z,x_1^d)$ satisfy $zf'(w_1)=f(x_1^d+z)-f(x_1^d)$ and $-zf'(w_2)=f(x_1^d-z)-f(x_1^d)$, respectively.
Moreover, 
$\left|f'(w_1)-f'(w_2)\right|\leq 2z|f''(w_3)|$ 
holds for some $w_3$ in the interval $(x_1^d-z,x_1^d+z)$. Since $x\mapsto 1\wedge e^x$ is Lipschitz with constant $1$, we get 
$$\left|\tilde{\beta}(\mathbf{x}^d,x^d_1+z)-\tilde{\beta}(\mathbf{x}^d,x^d_1-z)\right|\leq \left|\log(\rho)(x_1^d+z)-\log(\rho)(x_1^d-z)\right|
\leq 2z|\log(\rho)'(w_4)|$$
for some $w_4\in (x_1^d-z,x_1^d+z)$.  
By adding and subtracting 
$f'(w_2)\tilde{\beta}(\mathbf{x}^d,x^d_1+z)$
on the right-hand side of~\eqref{eq:auxGdBound}, applying the two bounds we just derived and noting that $\tilde \beta\leq 1$, 
we get
\begin{eqnarray}
\label{eq:d_scaling_inequality}
\left|\mathcal{G}_df(\mathbf{x}^d)\right|&\leq& 2d\int_0^\infty z^2|f''(w_3)|q(z)dz+ 2d\int_0^\infty z^2|f'(w_2)\log(\rho)'(w_4)|q(z)dz.
\end{eqnarray}
Note that, 
since
$\max\{|w_3|,|w_2|,|w_4|\} \leq |x_1^d|+z$ and $\|f''\|_{\infty,1}\leq \|f''\|_{\infty,1/2}$,
we have
$$|f''(w_3)|\leq \|f''\|_{\infty,1} e^{|w_3|} \leq \|f''\|_{\infty,1/2} e^{|x_1^d|+z},\quad
\log(\rho)'(w_4)f'(w_2)\leq \|\log(\rho)'\|_{\infty,1/2}\|f'\|_{\infty,1/2}e^{|x_1^d|+z},$$
which, together with inequality~\eqref{eq:d_scaling_inequality}, implies the lemma. 
\end{proof}

\begin{proof}[Proof of Proposition~\ref{thm:Lp}]
By Theorem~\ref{thm:GdG} (on $\mathcal{A}_d$) and Lemma~\ref{lemma:AuxLp1} (on $\RR^d\setminus \mathcal{A}_d$), there exists a constant 
$C>0$ such that for any  $f\in\mathcal{S}^3$ the following inequality holds:
\begin{multline*}
\left\|\mathcal{G}_df-\mathcal{G}f\right\|^p_p=
\int_{\mathcal{A}_d}\left|\mathcal{G}_df(\mathbf{x}^d)-\mathcal{G}f(\mathbf{x}^d)\right|^p\rho_d(\mathbf{x}^d)d\mathbf{x}^d+
\int_{\RR^d\setminus \mathcal{A}_d}\left|\mathcal{G}_df(\mathbf{x}^d)-\mathcal{G}f(\mathbf{x}^d)\right|^p\rho_d(\mathbf{x}^d)d\mathbf{x}^d\\
\leq C\rho(e^{p|x|})\left(\frac{a_d^p}{d^{p/2}}\rho_d(\mathcal{A}_d)+
\rho_d(\RR^d\setminus \mathcal{A}_d)\right)\left(\sum_{i=1}^3\|f^{(i)}\|_{\infty,1/2}\right)^p\text{.}
\end{multline*}
Apply Proposition~\ref{prop:AdSizes} and raise both sides of the inequality  
to the power $1/p$ to conclude the proof of the proposition.
\end{proof}

\subsection{Proof of Theorem~\ref{thm:AVbound}}
\label{subsec:proofthm3}

\begin{lem}\label{lem:GE}
Assume that $\rho$ is a strictly positive density in 
$\mathcal{C}^1$ and that~\eqref{eq:SupperExponentialTails} holds. Then, for any $d\in\NN$, the RWM chain $\{\mathbf{X}^d_n\}_{n\in\NN}$ is
$V$-uniformly ergodic with $V:=1/\sqrt{\rho_d}$.  
\end{lem}

\begin{proof}
The lemma follows from \cite[Theorem~4.1]{jarner} if we prove that the target $\rho_d$ satisfies
\begin{equation}\label{eq:jarner1}
\lim_{|\mathbf{x}^d|\to\infty}\frac{\mathbf{x}^d}{|\mathbf{x}^d|}\cdot\nabla\log(\rho_d(\mathbf{x}^d))=
\lim_{|\mathbf{x}^d|\to\infty}\sum_{i=1}^d\frac{x^d_i}{|x^d_i|}\log(\rho)'(x^d_i)=-\infty,
\end{equation}
\begin{equation}\label{eq:jarner2}
\liminf_{|\mathbf{x}^d|\to\infty} \mathbb{P}_{\mathbf{Y}^d}\left[\rho_d(\mathbf{Y}^d)\geq \rho_d(\mathbf{x}^d)\right]>0\text{.}
\end{equation}
Assumption~\eqref{eq:SupperExponentialTails} implies that the expression $x/|x|\cdot \log(\rho)'(x)$ 
is bounded above and takes arbitrarily large negative values as $|x|\to\infty$.
This yields~\eqref{eq:jarner1}, since $|\mathbf{x}^d|\to\infty$ implies that $|x^d_i|\to\infty$ holds for at least one
$i\in\{1,\ldots,d\}$.

Condition~\eqref{eq:jarner2}
states that the acceptance probability in the RWM chain is bounded away from zero sufficiently far from the origin.
To prove this, 
recall that 
$\mathbf{Y}^d\sim N(\mathbf{x}^d,l^2/d\cdot I_d)$
and
define the set 
$$B(\mathbf{x}^d):=\left\{\mathbf{y}^d\in\RR^d\colon
\frac{x^d_i}{|x^d_i|}\cdot(y^d_i-x^d_i)\in
\left(\frac{-2l}{\sqrt{d}},\frac{-l}{\sqrt{d}}\right)\text{ for all }i\leq
d\right\}\text{,}$$ where we interpret $x^d_i/|x^d_i|:=1$ if $x^d_i=0$. Clearly
$\inf_{\mathbf{x}^d\in\RR^d}\mathbb{P}_{\mathbf{Y}^d}\left[B(\mathbf{x}^d)
\right]>0$. 
We now prove that if $|\mathbf{x}^d|$ is sufficiently large, then  
$\rho_d(\mathbf{y}^d)\geq\rho_d(\mathbf{x}^d)$ for all $\mathbf{y}^d\in B(\mathbf{x}^d)$,
which implies~\eqref{eq:jarner2}.

By~\eqref{eq:SupperExponentialTails}, far enough from zero,
$\rho$ is decreasing in a direction away from the origin. Therefore, there
exists a compact interval $K\subset \RR$ such that
$(-2l/\sqrt{d},2l/\sqrt{d})\subset K$ and $\rho(y)\geq \rho(x)$ whenever
$x\notin K$ and $x/|x|\cdot (y-x)\in (-2l/\sqrt{d},-l/\sqrt{d})$. 
We claim that 
for every $\mathbf{y}^d\in B(\mathbf{x}^d)$, 
the inequality 
$\rho(y^d_i)/\rho(x^d_i)\geq (\min_{x\in K}\rho(x))/(\max_{x\in\RR}\rho(x))\in(0,1)$
holds. 
If
$x^d_i\in K$, 
then
$y^d_i\in K$
and the inequality follows trivially. 
If
$x^d_i\notin K$,
then, by the definition of $K$, 
we have
$\rho(y^d_i)/\rho(x^d_i)\geq 1$. 
This proves the claim. 
Hence, for $\mathbf{y}^d\in B(\mathbf{x}^d)$ 
we have
\begin{equation}
\label{eq:good_estimate_rho_d}
\frac{\rho_d(\mathbf{y}^d)}{\rho_d(\mathbf{x}^d)}\geq
\left(\max_{i\leq d}\frac{\rho(y^d_i)}{\rho(x^d_i)}\right)\cdot\left(\frac{\min_{x\in
K}\rho(x)}{\max_{x\in\RR}\rho(x)}\right)^{d-1}\text{.}
\end{equation}

We now prove that the ratio
$\rho(y^d_i)/\rho(x^d_i)$ takes arbitrarily large values as $|x^d_i|\to\infty$. 
To show this,  
pick $\mathbf{y}^d\in B(\mathbf{x}^d)$ 
and assume the inequality $y_i^d>x_i^d$. 
Then  
$x_i^d<0$ and 
$y_i^d-x_i^d>l/\sqrt{d}$.
Moreover the following holds 
$$\frac{\rho(y^d_i)}{\rho(x^d_i)}=\exp\left(\log\left(\frac{\rho(y^d_i)}{\rho(x^d_i)}\right)\right)\geq
1+\int_{x^d_i}^{y^d_i}\log(\rho)'(z)dz\geq
1+l/\sqrt{d}\inf_{z<x^d_i+2l/\sqrt{d}}\log(\rho)'(z)\to\infty$$
as $x_i^d\to-\infty$ by~\eqref{eq:SupperExponentialTails}. 
This, together with~\eqref{eq:good_estimate_rho_d}, implies~\eqref{eq:jarner2}.  
The case $y_i^d<x_i^d$ is analogous and the lemma follows.
\end{proof}

\begin{prop}\label{prop:fhatinS}
If a strictly positive $\rho$ 
satisfies~\eqref{eq:SupperExponentialTails}
and
$\log(\rho)\in\mathcal{S}^{n_\rho}$ and $f\in\mathcal{S}^{n_f}$ for some
integers $n_\rho,n_f\in\NN\cup\{0\}$, then the function $\hat f$, defined in~\eqref{eq:solutions_Poisson_Eq_diff},
satisfies
$\hat{f}\in\mathcal{S}^{\min(n_f+2,n_\rho+1)}$.  
\end{prop}

\begin{proof}
Clearly, if $f\in\mathcal{C}^{n_f}$ and $\rho\in\mathcal{C}^{n_\rho}$ and if
$\rho$ is strictly positive, then
$\hat{f}\in\mathcal{C}^{\min(n_f+2,n_\rho+1)}$.  Pick $s>0$. 
The L'Hospital's rule implies:
$$\lim_{x\to\infty}\frac{\hat{f}(x)}{e^{s|x|}}=
\frac{2}{sh(l)}\lim_{x\to\infty}\frac{\int_{-\infty}^x\rho(y)(\rho(f)-f(y))dy}{e^{sx}\rho(x)}=\frac{2}{sh(l)}\lim_{x\to\infty}\frac{\rho(f)-f(x)}{se^{sx}+e^{sx}(\log(\rho))'(x)}\text{.}$$
The last limit is zero by~\eqref{eq:SupperExponentialTails}. An analogous argument shows 
$\lim_{x\to-\infty}\hat{f}(x)/e^{s|x|}=0$. 
Hence $\|\hat{f}\|_{\infty,s}<\infty$ holds for all $s>0$. 
Since $h(l)\hat{f}'(x)/2=\left(\int_{-\infty}^x\rho(y)(\rho(f)-f(y))dy\right)/\rho(x)$, 
this argument implies that $\|\hat{f}'\|_{\infty,s}<\infty$ holds for all $s>0$. 
Hence $\hat{f}\in \mathcal{S}^{1}$.

Proceed by induction: assume that for all $k\leq n$ (where $1\leq
n<\min(n_f+2,n_\rho+1)$) we have $\|\hat{f}^{(k)}\|_{\infty,s}<\infty$
for any $s>0$.
Pick an arbitrary $u>0$. 
By differentiating \eqref{eq:UfPE} we obtain
$$\hat{f}^{(n+1)}=-\sum_{k=0}^{n-1}\binom{n-1}{k}(\log(\rho))^{(k+1)}\hat{f}^{(n-k)}+\frac{2}{h(l)}(\rho(f)-f)^{(n-1)}\text{.}$$
Since $n\leq \min(n_\rho, n_f+1)$, the induction hypothesis implies
$\|\hat{f}^{(k)}\|_{\infty,u/2}<\infty$ for all $1\leq k\leq n$.
By assumption we have 
$\|f^{(n-1)}\|_{\infty,u}<\infty$ and 
$\|(\log(\rho))^{(k)}\|_{\infty,u/2}<\infty$
for all $1\leq k\leq n$.
Hence
$\|\hat{f}^{(n+1)}\|_{\infty,u}<\infty$ holds for an arbitrary $u>0$ and the proposition follows. 
\end{proof}

\begin{proof}[Proof of Theorem~\ref{thm:AVbound}]
By Lemma~\ref{lem:GE},
the RWM chain $\mathbf{X}^d$ with the transition kernel 
$P_d$ is $V$-uniformly ergodic with
$V=\rho_d^{-1/2}$. 
Moreover, by~\cite{RobertsRosenthal97}[Prop.~2.1 and Thm~2.1], $P_d$ defines
a self-adjoint operator on $\{g\in L^2(\rho_d)\colon \rho_d(g)=0\}$ with norm
$\lambda_d<1$.
Proposition~\ref{prop:fhatinS}
implies 
$\hat{f}\in\mathcal{S}^3$,
since by assumption we have
$f\in\mathcal{S}^1$ and $\log(\rho)\in\mathcal{S}^4$. 
By Remark~\ref{rema:SnProperties}(c) in Section~\ref{sec:Technical} below we have
$\hat{f}^2\in\mathcal{S}^3$.
Since $P_d\hat f = (1/d)\mathcal{G}_d\hat f+\hat f$, Lemma~\ref{lemma:AuxLp1} implies that
$(P_d\hat{f})^2(\mathbf{x}^d)\leq C_{\hat{f}} e^{2 |x_1^d|}$ for some positive
constant $C_{\hat{f}}$ and all $\mathbf{x}^d\in\RR^d$. 
Hence~\eqref{eq:SupperExponentialTails} and the definition of $V$
imply  the inequality 
$\max\{\hat{f}^2,(P_d\hat{f})^2\}\leq cV$ for some constant $c>0$. 
Consequently, by \cite[Theorem 17.0.1]{tweedie}, the CLT for the chain 
$\mathbf{X}^d$ and function $f+dP_d\hat{f}-d\hat{f}$ holds with some
asymptotic variance $\hat \sigma^2_{f,d}$.

By~\cite{KipnisVaradhan86,Geyer92} we can represent $\hat \sigma^2_{f,d}$ 
in terms of a positive spectral measure $E_d(d\lambda)$ associated with the
function
$f-\rho(f)+dP_d\hat{f}-d\hat{f}=\mathcal{G}_d\hat{f}-\mathcal{G}\hat{f}$ as
$$\hat \sigma^2_{f,d}=\int_{\Lambda_d}\frac{1+\lambda}{1-\lambda}E_d(d\lambda),$$
where $\Lambda_d\subset[-\lambda_d,\lambda_d]$ denotes the spectrum of the self-adjoint operator $P_d$ acting on the Hilbert space
$\{g\in L^2(\rho_d)\colon \rho_d(g)=0\}$. 
By the definition of the spectral measure $E_d(d\lambda)$ we obtain
We can bound
$$\hat \sigma^2_{f,d}\leq\frac{1+\lambda_d}{1-\lambda_d}\int_{\Lambda_d}E_d(d\lambda)=\frac{1+\lambda_d}{1-\lambda_d}
\|P_d(d\hat{f})-d\hat{f}+f-\rho(f)\|^2_2\leq
\frac{2}{1-\lambda_d}\|\mathcal{G}_d\hat{f}-\mathcal{G}\hat{f}\|^2_2\text{.}$$
Finally, the result follows by Proposition~\ref{cor:Lpnorm}. 
\end{proof}

\section{Technical results}\label{sec:Technical}
The results in Section~\ref{sec:Technical} use the ideas of Berry-Esseen theory and large deviations as well as the
optimal Young inequality, and do not depend on anything in this paper that precedes them. 

\subsection{Bounds on the expectations of test functions}\label{sec:TestFunBounds}
We start with elementary observations. 

\begin{rema}\label{rema:SnProperties}
Recall that $\mathcal{S}^n$, $n\in\NN\cup\{0\}$, is defined in~\eqref{eq:defSn}. The following statements hold.
\begin{enumerate}[(a)]
\item If $n\leq m$, then $\mathcal{S}^{m}\subset \mathcal{S}^n$.
\item For $n\in\NN$, $f\in\mathcal{S}^n$ if and only if $f'\in\mathcal{S}^{n-1}$.
\item If $f\in\mathcal{S}^n$ and $g\in\mathcal{S}^m$ then $f+g, fg\in\mathcal{S}^{\min(n,m)}$.
\end{enumerate}
\end{rema}

\begin{prop}\label{prop:TaylorEstimate} Pick an arbitrary $n\in\NN$. Assume $f\in\mathcal{S}^n$, $k\leq n$,
$x\in\RR$ and $Y\sim N(x,\sigma^2)$. 
Then there exists measurable $Z$ satisfying
$f^{(k)}(Z)(Y-x)^k/k!=f(Y)-\sum_{i=0}^{k-1}f^{(i)}(x)(Y-x)^i/i!$ and 
$|Z-x|<|Y-x|$. Furthermore  
there exists a constant $C>0$
(depending on $n$) such that,
for any $m\in\NN$ and $s>0$ we have
$$\mathbb{E}_{Y}\left[\left|f^{(k)}(Z)\right|^{m}\left|Y-x\right|^{n}\right]\leq
Ce^{s^2\sigma^2}\mathbb{E}_{Y}\left[\left|Y-x\right|^{n}\right]\|f^{(k)}\|_{\infty,s/m}^me^{s|x|}.$$
\end{prop}

\begin{proof}
A random variable $Z$, defined via the integral form of the remainder in Taylor's theorem, lies a.s. between $Y$ and $x$, implying 
$|Z-x|<|Y-x|$. 
Cauchy's inequality yields 
\begin{equation}
\label{eq:Simple_CS_bound_normal}
\mathbb{E}_{Y}\left[\left|f^{(k)}(Z)\right|^{m}\left|Y-x\right|^{n}\right]^2\leq
\mathbb{E}_{Y}\left[\left|f^{(k)}(Z)\right|^{2m}\right]\mathbb{E}_{Y}\left[\left|Y-x\right|^{2n}\right]\text{.}
\end{equation}
Since $f\in\mathcal{S}^n\subset\mathcal{S}^k$, we have
$\sup_{x\in\RR}\left|f^{(k)}(x)\right|^{2m}e^{-2s|x|}=\|f^{(k)}\|^{2m}_{\infty,s/m}<\infty$.
As $Y\sim N(x,\sigma^2)$, the equality 
$\mathbb{E}_{Y}\left[\left|Y-x\right|^{2n}\right]= C^2
\mathbb{E}_{Y}\left[\left|Y-x\right|^{n}\right]^2$
holds,
where
$C:=(2\sqrt{\pi}\Gamma((2n+1)/2))^{1/2}/\Gamma((n+1)/2)$ 
and
$\Gamma(\cdot)$ is the Euler gamma function. 
Hence, by~\eqref{eq:Simple_CS_bound_normal},
we get
$$\mathbb{E}_{Y}\left[\left|f^{(k)}(Z)\right|^{m}\left|Y-x\right|^{n}\right]\leq
\frac{C}{\sqrt{2}}\|f^{(k)}\|^{m}_{\infty,s/m}\sqrt{\mathbb{E}_{Y}\left[e^{2s|Z|}\right]}\mathbb{E}_{Y}\left[\left|Y-x\right|^{n}\right]\text{.}$$
It remains to
note 
$\mathbb{E}_{Y}e^{2s(|Z|-|x|)}\leq
\mathbb{E}_{Y}e^{2s|Z-x|}\leq
\mathbb{E}_{Y}e^{2s|Y-x|}\leq2 \mathbb{E}_{Y}e^{2s(Y-x)}= 2e^{2s^2\sigma^2}$.
\end{proof}

\begin{prop}\label{prop:BoundingSums}
Let $f\colon\RR\to\RR$ be a measurable (not necessary continuous)
function such that $\|f\|_{\infty,1/2}<\infty$. 
 Fix $n\in\NN$,
$\mathbf{x}^d\in\RR^d$ and let $X_1, X_2\dots, X_d$ be IID copies of $X$,
satisfying $\mathbb{E}\left[X^n\right]=0$ and
$\mathbb{E}\left[X^{2n}\right]<\infty$. Then the following inequality holds: 
$$\left|\mathbb{E}\left[\sum_{i=1}^df(x^d_i)X_i^n\right]\right|\leq
\|f\|_{\infty,1/2}
\left(\mathbb{E}[X^{2n}]\sum_{i=1}^de^{|x^d_i|}\right)^{1/2}\text{.}$$
\end{prop}

\begin{rema}
Note that the assumptions of Proposition~\ref{prop:BoundingSums} imply that, if $X$ is a non-zero random variable, then  $n\in\NN$
has to be odd. 
\end{rema}

\begin{proof}
By Jensen's inequality, the fact that $\mathbb{E}[X]=0$ and the assumption on $f$ we get 
$$\mathbb{E}\left[\sum_{i=1}^df(x^d_i)X_i^n\right]^2\leq
\mathbb{E}\left[\left(\sum_{i=1}^df(x^d_i)X_i^n\right)^2\right]=
\sum_{i=1}^d (f(x^d_i))^2\mathbb{E}\left[X_i^{2n}\right]\leq\|f\|^{2}_{\infty,1/2} \mathbb{E}[X^{2n}]\sum_{i=1}^de^{|x^d_i|}\text{.}$$
\end{proof}

\subsection{Deviations of the sums of IID random variables}\label{sec:EventsBounds}

\begin{prop}
\label{prop:Probabilities&Sets} Let $f\in \mathcal{S}^0$ be such
that $\rho(f)=0$ and let $a=\{a_d\}_{d\in\NN}$ be a sluggish sequence. If the random vector 
$(X_{1,d},\ldots, X_{d,d})$ follows the density $\rho_d$ for all $d\in\NN$, then for every $t>0$ the
following inequality
holds for all but finitely many $d\in\NN$:
$$ \mathbb{P}_{\rho_d}\left[\left|\frac{1}{d-1}\sum_{i=2}^df(X_{i,d})\right|\geq \frac{ta_d}{\sqrt{d}}\right]\leq \exp(-t^2a_d^2/(3\rho(f^2))).$$
\end{prop}

\begin{rema}\label{remark:LargeDev}
Proposition~\ref{prop:Probabilities&Sets} is an elementary consequence
of a deeper underlying result, that the sequence of random variables
$\{\sum_{i=1}^df(X_{i,d})/(a_d\sqrt{d})\}_{d\in\NN}$ 
satisfies a moderate deviation principle with a good rate function $t\mapsto t^2/(2\rho(f^2))$ and speed $a_d^2$
(see~\cite{EichelsbacherLowe02} for details). The key inequality needed in the proof of 
Proposition~\ref{prop:Probabilities&Sets} is given in the next lemma. 
\end{rema}

\begin{lem}
\label{lem:LargeDev}
Let assumptions of Proposition~\ref{prop:Probabilities&Sets} hold. If $\rho(f^2)>0$, then for every closed $F\subseteq \RR$ 
the following holds:
$$\limsup_{d\to\infty}a^{-2}_d\log\mathbb{P}_{\rho_d}\left[\sum_{i=1}^df(X_{i,d})/(a_d\sqrt{d})\in
F\right]\leq -\inf\{ x^2/(2\rho(f^2));x\in F\}\text{.} $$
\end{lem}

\begin{proof} The moderate deviations results~\cite[Thm~2.2, Lem.~2.5, Rem.~2.6]{EichelsbacherLowe02}  
yield a sufficient condition for the above inequality.  
More precisely, for $X\sim \rho$, 
we need to establish: 
\begin{equation}\label{eq:LDAuxilliaryCondition}
\limsup_{d\to\infty} a_d^{-2}\log \left(d\cdot \mathbb{P}_{\rho}\left[|f(X)|\geq a_d\sqrt{d}\right]\right)=-\infty\text{.}
\end{equation}
Fix an arbitrary $m\in\NN$. Since $f\in\mathcal{S}^0$, we have $|f(x)|\leq
\|f\|_{\infty,1/m} e^{|x|/m}$ for every $x\in\RR$. 
Consequently, for all large $d$, we get 
$$\mathbb{P}_{\rho}\left[|f(X)|\geq a_d\sqrt{d}\right]\leq
\mathbb{P}_{\rho}\left[\|f\|_{\infty,1/m}^m e^{|X|}\geq d^{m/2}\right]\leq
\|f\|_{\infty,1/m}^m \rho\left(e^{|X|}\right)d^{-m/2}.$$
Since $\{a_d\}_{d\in\NN}$ is sluggish,
$\exists C_0>0$ such that
$a_d^{-2}\log(\|f\|_{\infty,1/m}^m \rho(e^{|X|}))<C_0< a_d^{-2}\log(d)$ for all large $d\in\NN$.
Hence
$$a_d^{-2}\log (d\cdot\mathbb{P}_{\rho}[|f(X)|\geq a_d\sqrt{d}])\leq
a_d^{-2}(\log(\|f\|_{\infty,1/m}^m \rho(e^{|X|}))-(m/2-1)\log(d))<
-C_0(m/2-2),$$
for all large $d\in\NN$.
Since $m$ was arbitrary,~\eqref{eq:LDAuxilliaryCondition} follows. 
\end{proof}

\begin{proof}[Proof of Proposition~\ref{prop:Probabilities&Sets}]
Note that the proposition holds if $\rho(f^2)=0$. Assume now $\rho(f^2)>0$
and fix an arbitrary $t>0$. Note that since $\{a_d\}_{d\in\NN}$ is sluggish, 
so is $\{a'_d\}_{d\in\NN}$, $a'_d:=a_{d+1}\sqrt{d/(d+1)}$. 
Apply Lemma~\ref{lem:LargeDev} to
$F=\RR\setminus (-t,t)$ and 
$\{a'_d\}_{d\in\NN}$
to get the following inequality 
\begin{equation}
\label{eq:From_LDP}
\mathbb{P}_{\rho_{d-1}}\left[\left|\sum_{i=1}^{d-1}f(X_{i,d-1})/(a'_{d-1}\sqrt{d-1})\right|\geq
t\right]\leq \exp\left(-3 (a'_{d-1})^2t^2/(8\rho(f^2))\right)
\end{equation}
for all large enough $d\in\NN$. 
Since
$3(a'_{d-1})^2/4\geq 2a^2_d/3$ 
for all but finitely many $d\in\NN$,
the right-hand side in~\eqref{eq:From_LDP} is bounded above by $\exp(-(a_d)^2t^2/(3\rho(f^2)))$. 
Recall
$\rho_d(\mathbf{x}^d)=\rho_{d-1}(\mathbf{x}^{d-1})\rho(x_d^d)$
and $a_{d-1}'\sqrt{d-1}=a_d(d-1)/\sqrt{d}$.
Hence the left-hand side in inequality~\eqref{eq:From_LDP} equals
$\mathbb{P}_{\rho_d}[|\sum_{i=2}^{d}f(X_{i,d})/(d-1)|\geq ta_d/\sqrt{d}]$
and the proposition follows. 
\end{proof}

The next result is based on a combinatorial argument. A special case 
of Proposition~\ref{prop:AdSizesAux} was used in~\cite{RobertsGelmanGilks97}.

\begin{prop}\label{prop:AdSizesAux}
Let $n\in\NN$ and a measurable $f\colon\RR\to\RR$ satisfy $\rho(f)=0$ and
$\rho(f^{2n})< \infty$. 
If the random vector $(X_{1,d},\ldots,X_{d,d})$ is distributed according to $\rho_d$,
then there exists a constant $C$, independent of $d$, such that
$\mathbb{P}_{\rho_d}\left[\left|\frac{1}{d-1}\sum_{i=2}^df(X_{i,d})\right|\geq
1\right]\leq C d^{-n}$.  
\end{prop}

\begin{rema}
The constant $C$ in Proposition~\ref{prop:AdSizesAux} may depend on $n\in\NN$ and the function $f$.
\end{rema}

\begin{proof}
Fix $n\in\NN$
and let $\NN_0:=\NN\cup\{0\}$.
Markov's inequality and the Multinomial theorem yield: 
\begin{eqnarray}
&&\mathbb{P}_{\rho_d}\left[\left|\frac{1}{d-1}\sum_{i=2}^df(X_{i,d})\right|\geq
1\right]=\mathbb{P}_{\rho_d}\left[\left|\frac{1}{d-1}\sum_{i=2}^d
f(X_{i,d})\right|^{2n}\geq 1\right]\leq
\mathbb{E}_{\rho_d}\left(\frac{1}{d-1}\sum_{i=2}^d
f(X_{i,d})\right)^{2n}\nonumber\\
&&=(d-1)^{-2n}\sum_{\substack{k_2+k_3+\cdots
+k_d=2n\\k_2,k_3\dots,k_d\in\NN_0\setminus
\{1\}}}\binom{2n}{k_2,k_3,\dots,k_d}\prod_{i=2}^d\mathbb{E}_{\rho}\left[f(X_{i,d})^{k_i}\right]\nonumber\text{,}
\end{eqnarray}
where last equality holds, because the expectation of any summand of the form
$\prod_{i=2}^{d}f(X_{i,d})^{k_i}$ is zero if any of the indices $k_i=1$ since $\rho_d$ has a product structure and
$\rho(f)=0$.  
By Jensen's inequality, 
$\prod_{i=2}^d\mathbb{E}_{\rho}\left[f(X_{i,d})^{k_i}\right]\leq
\prod_{i=2}^d\mathbb{E}_{\rho}\left[f(X_{i,d})^{2n}\right]^{k_i/2n}=\rho(f^{2n})
^{\sum_{i=2}^{d}\frac{k_i}{2n}}=\rho(f^{2n})$, and hence  
\begin{equation}
\label{eq:Only_comb_arg_needed}
\mathbb{P}_{\rho_d}\left[\left|\frac{1}{d-1}\sum_{i=2}^df(X_{i,d})\right|\geq
1\right]\leq (2n)!\cdot\rho(f^{2n})(d-1)^{-2n}\cdot
\left|\mathcal{N}_d\right|\text{,}
\end{equation}
where $\left|\mathcal{N}_d\right|$ stands
for the cardinality of the set 
$$\mathcal{N}_d:=\left\{(k_2,k_3,\dots, k_d)\in\NN^{d-1}_0;\quad\sum_{i=2}^dk_d=2n\text{ and } k_i\neq 1\text{ for all 
} 2\leq i\leq d\right\}\text{.}$$
Inequality~\eqref{eq:Only_comb_arg_needed} and the next Claim prove the proposition. 

\noindent \textbf{Claim.} 
$\left|\mathcal{N}_d\right|\leq C' d^n$ for a constant $C'$ independent of $d$. 

\noindent \textit{Proof of Claim.}
Consider a function $\zeta\colon \NN^{d-1}_0\to\NN^{d-1}_0$, 
$\zeta(a_2,a_3,\dots, a_d):=(2\lfloor\frac{a_2}{2}\rfloor,2\lfloor\frac{a_3}{2}\rfloor,\dots
2\lfloor\frac{a_d}{2}\rfloor )$, that rounds each entry down to the
nearest even number. Every element in the image $\zeta(\mathcal{N}_d)$ is a
$(d-1)$-tuple of non-negative even integers with sum at most $2n$.
Recall the number of $k$-combinations with repetition, chosen from a set of
$d-1$ objects, equals $\binom{k+d-2}{k}$. There exists $C''>0$, such that  
\begin{eqnarray}
\left|\zeta(\mathcal{N}_d)\right|&\leq& \left|\left\{(k_2,k_3,\dots, k_d)\in\NN^{d-1}_0; \quad\sum_{i=2}^dk_d\leq n\right\}\right|\nonumber\\
&=&\sum_{k=0}^n\left|\left\{(k_2,k_3,\dots, k_d)\in\NN^{d-1}_0;\quad\sum_{i=2}^dk_d=k\right\}\right|=\sum_{k=0}^n\binom{k+d-2}{k}\leq C'' d^n\text{.}\nonumber
\end{eqnarray}

Note that the pre-image of a singleton under $\zeta$ contains at most $2^n$ elements (i.e. $(d-1)$-tuples) of $\mathcal{N}_d$. 
Indeed, by the definition of
$\mathcal{N}_d$, at most $n$ coordinates of an element are not zero and each can
either reduce by one or stay the same. Hence, for $C':=C'' 2^n$, we have
$\left|\mathcal{N}_d\right|\leq 2^n\cdot\left|\zeta(\mathcal{N}_d)\right|\leq C'd^n$. 
\end{proof}

\subsection{Bounds on the densities of certain random variables}\label{sec:PDFBOunds}

The key step in the proof of Proposition~\ref{prop:SumPolynomialPDFboun} below is 
the optimal Young's inequality:
for $p,q\geq 1$ and $r\in[1,\infty]$, such that  $1/p+1/q=1+1/r$, 
and functions $f\in L^p(\RR)$ and $g\in L^q(\RR)$, their convolution  $f*g$
satisfies the inequality
\begin{equation}\label{eq:OptimalYoung}
\|f*g\|_r\leq \frac{C_pC_q}{C_r}\|f\|_p\|g\|_q\text{,}\quad\text{where}\quad 
C_s:=\begin{cases}
\sqrt{\frac{s^{1/s}}{s'^{1/s'}}},
\text{ if $s\in(1,\infty)$ and $1/s+1/s'=1$,}\\
1\text{, if $s\in\{1,\infty\}$.}
\end{cases}
\end{equation}
For $s\in[1,\infty)$,
$\|\cdot\|_s$ is the usual norm on $L^s(\RR)$ and $\|\cdot\|_\infty$
denotes the essential supremum norm on $L^\infty(\RR)$. 
The proof of~\eqref{eq:OptimalYoung} for $r<\infty$ is given in~\cite[Thm~1]{Barthe98}.
In the case $r=\infty$, we have $C_pC_q/C_r=1$ and the inequality in~\eqref{eq:OptimalYoung}
follows from the definition of the convolution, translation invariance of the Lebesgue measure 
and H\"older's inequality. 

\begin{prop}\label{prop:SumPolynomialPDFboun} 
Let $X_1,X_2,\dots, X_d$ be
independent random variables, each $X_i$ with a bounded density $q_i$. The density
$Q_d$ of the sum $\sum_{i=1}^d X_i$ satisfies
$\|Q_d\|_\infty\leq c  \max_{i\leq d}\|q_i\|_\infty/\sqrt{d}$
for some constant $c>0$.
\end{prop}

\begin{rema}
The factor $d^{-1/2}$  in the inequality of Proposition~\ref{prop:SumPolynomialPDFboun}
above comes from~\eqref{eq:OptimalYoung}
and is crucial for the analysis in this paper. The standard Young's inequality for convolutions would only yield 
$\|Q_d\|_\infty\leq c \max_{i\leq d}\|q_i\|_\infty$, which gives insufficient control over $Q_d$. 
\end{rema}

\begin{proof}[Proof of Proposition~\ref{prop:SumPolynomialPDFboun}]
Since random variables $X_i$ are independent, the density of their sum is a
convolution of the respective densities, $Q_d=\Asterisk_{i=1}^d q_i$.
For all $i$ and each $t>1$ we have 
$q_i\in L^{\infty}(\RR)\cap L^{1}(\RR)\subset L^{t}(\RR)$. 
Moreover, the following inequality holds 
for every $k\leq d-1$:
\begin{equation}\label{eq:IndHypothesisYI}
\left\|Q_d\right\|_{\infty}=\left\|\Asterisk_{i=1}^d q_i\right\|_{\infty}\leq
\left(C_{\frac{d}{d-1}}\right)^kC_{\frac{d}{k}}\left(\prod_{i=1}^k\|q_i\|_{\frac{d}{d-1}}\right)\left\|\Asterisk_{i=k+1}^d
q_i\right\|_{\frac{d}{k}}
\end{equation}

We prove~\eqref{eq:IndHypothesisYI}  by induction on $k$. 
For $k=1$, note that 
$d$ and $\frac{d}{d-1}$ are H\"older conjugates, i.e. $1/d +1/(d/(d-1))=1$
Hence~\eqref{eq:OptimalYoung} for $r=\infty$, $q=d$, $p=d/(d-1)$, $f=q_1$
and 
$g=\Asterisk_{i=2}^d q_i$ implies
$\|Q_d\|_{\infty}\leq \|q_1\|_{\frac{d}{d-1}}\left\|\Asterisk_{i=2}^d q_i\right\|_{d}$ 
and $C_{\frac{d}{d-1}}=C^{-1}_d$.
Now assume~\eqref{eq:IndHypothesisYI} holds for some $k\leq d-2$. 
Since
$\left(d/(d-1)\right)^{-1}+\left(d/(k+1)\right)^{-1}=1+\left(d/k\right)^{-1}$,
the inequality in~\eqref{eq:OptimalYoung} implies $$\left\|\Asterisk_{i=k+1}^d q_i\right\|_{\frac{d}{k}}\leq
\frac{C_{\frac{d}{d-1}}C_{\frac{d}{k+1}}}{C_{\frac{d}{k}}}\|q_{k+1}\|_{\frac{d}{d-1}}\left\|\Asterisk_{i=k+2}^d q_i\right\|_{\frac{d}{k+1}}\text{.}$$
This inequality and the induction hypothesis (i.e.~\eqref{eq:IndHypothesisYI} for $k$) implies~\eqref{eq:IndHypothesisYI} for $k+1$.

Since $q_1$ is a density, we have $\|q_i\|_1=1$. Hence we find
$\|q_i\|_{\frac{d}{d-1}}\leq \|q_i\|_1^{\frac{d-1}{d}}\|q_i\|_{\infty}^{\frac{1}{d}}=\|q_i\|_{\infty}^{\frac{1}{d}}$
for each $i$, and in particular
$\prod_{i=1}^d\|q_i\|_{\frac{d}{d-1}}\leq \max_{i\leq d}\left(\|q_i\|_\infty\right)$.
By~\eqref{eq:IndHypothesisYI} for $k=d-1$ we get	
$$\left\|Q_d\right\|_{\infty}\leq \left(C_{\frac{d}{d-1}}\right)^d\prod_{i=1}^d\|q_i\|_{\frac{d}{d-1}}\leq 
\max_{i\leq d}\left(\|q_i\|_\infty\right)\left(C_{\frac{d}{d-1}}\right)^d.$$
Since
 $\lim_{d\to\infty}\sqrt{d}\left(C_{\frac{d}{d-1}}\right)^d=\sqrt{e}$, 
there exists $c>0$ such that 
$\left(C_{\frac{d}{d-1}}\right)^d\leq c/ \sqrt{d}$ for all $d\in\NN$. 
\end{proof}
Polynomials of continuous random variables play an important role in the proofs of Section~\ref{sec:proofs}.

\begin{prop}\label{prop:p(N)pdf}
Let $X$ be a continuous random variable and $p$ a polynomial. Then the random variable $p(X)$ has a density.
\end{prop}

\begin{proof}
The set
$B:=p\left((p')^{-1}\left(\{0\}\right)\right)$ has finitely many points.
Moreover, $p$ is locally invertible on $\RR\setminus B$ by the inverse function theorem
and the inverses are differentiable. Hence, for any $x\notin B$, the set $p^{-1}\left((-\infty,x]\right)$
is a disjoint union of intervals with boundaries that depend smoothly on $x$.
Since $\mathbb{P}[p(X)\leq x]=\mathbb{P}[X\in p^{-1}\left((-\infty,x]\right)]$, the proposition follows. 
\end{proof}

\begin{prop}\label{prop:PolynomialPDFbound} 
Let $N=N(\mu,\sigma^2)$ be a normal
random variable and $p$ a polynomial satisfying $\inf_{x\in\RR}|p'(x)|\geq c_p$ for some constant $c_p>0$.
Then the random variable $p(N)$ has a probability density function $q_{p(N)}$,
which satisfies $\|q_{p(N)}\|_\infty\leq (c_p\sigma\sqrt{2\pi})^{-1}$.
\end{prop}

\begin{proof} Obviously, $p$ is strictly monotonic and thus a bijection.
Moreover, the distribution $\Phi_{p(N)}(\cdot)$ of $p(N)$ takes the form 
$\mathbb{P}\left[N\leq p^{-1}(\cdot)\right]$ or $\mathbb{P}\left[N>p^{-1}(\cdot)\right]$.
Hence, for any $x\in\RR$,  the density $q_{p(N)}$ of $p(N)$ satisfies
$q_{p(N)}(x)=q_{N}\left(p^{-1}(x)\right)\left|\left(p^{-1}\right)'(x)\right|
=q_N\left(p^{-1}(x)\right)/\left|p'\left(p^{-1}(x)\right)\right|\leq 1/(c_p\sigma\sqrt{2\pi})$,
as the density of $N$, $q_N$, 
is bounded above by 
$(\sigma\sqrt{2\pi})^{-1}$.  
\end{proof}

\subsection{CFs and distributions of near normal random variables}\label{sec:CDFandCFBounds}

\begin{prop}\label{prop:Berry}
Let $N$ be a normal random variable with mean $\mu$ and variance $\sigma^2$
and $X$ a continuous random variable. Denote with $\varphi_X$,
$\varphi_N$ and $\Phi_X$, $\Phi_N$ the CFs and the distributions of $X$ and
$N$, respectively. Assume there exist constants $r>0$, $\gamma\in (0,1)$ and a function
$R\colon\RR\to\RR$ such that
$\left|\log\varphi_X(t)-\log\varphi_N(t)\right|\leq R(t)\leq \gamma\sigma^2
t^2/2$ holds on $|t|\leq r$.  Then 
$$\sup_{x\in\RR}\left|\Phi_{N}(x)-\Phi_{X}(x)\right|\leq
\int_{-r}^{r}\frac{R(t)}{\pi |t|}\exp\left(-\frac{(1-\gamma)\sigma^2t^2}{2}\right)dt+\frac{12\sqrt{2}}{ \pi^{3/2}\sigma r}\text{.}$$ 
\end{prop}

\begin{rema}
The result is a direct consequence of the Smoothing theorem (see \cite[Theorem~2.5.2]{Kolassa}) commonly used to prove Berry-Esseen-type bounds, that relate CFs and distribution functions of random variables. 
\end{rema}

\begin{proof}
The Smoothing theorem implies 
$$\sup_{x\in\RR}\left|\Phi_{N}(x)-\Phi_{X}(x)\right|
\leq\int_{-r}^{r}\left|\varphi_{N}(t)-\varphi_X(t)\right|/(\pi|t|)dt+24\sup_{x\in\RR}|\Phi_N'(x)|/(\pi
r).$$ 
Note that, for any $z\in\mathbb{C}$, it holds 
$|e^z-1|\leq |z|e^{|z|}$. For 
$z:=\log(\varphi_X(t)/\varphi_N(t))$, this implies
$$|\varphi_X(t)-\varphi_N(t)|\leq |\varphi_N(t)| 
|\log \varphi_X(t)-\log \varphi_N(t)| \exp(|\log \varphi_X(t)-\log \varphi_N(t)|) \qquad \forall t\in\RR.
$$
The result follows from this inequality, $\sup_{x\in\RR}|\Phi_N'(x)|=1/(\sigma \sqrt{2\pi})$ and $|\varphi_N(t)|=e^{-\sigma^2t^2/2}$:
\begin{multline}
\int_{-r}^{r}{\frac{\left|\varphi_{N}(t)-\varphi_{X}(t)\right|}{\pi |t|}dt}\leq
\int_{-r}^{r}{|\varphi_{N}(t)|\frac{R(t)}{\pi |t|}e^{R(t)}dt} 
\leq \int_{-r}^{r}\frac{R(t)}{\pi
|t|}\exp\left(-\frac{(1-\gamma)\sigma^2t^2}{2}\right)dt\text{.}\nonumber
\end{multline}

\end{proof}

\begin{lem}\label{prop:NearNormalChar}
Let $X$ be random variable with finite mean $\mu$, variance $\sigma^2$ and
absolute third central moment
$\kappa:=\mathbb{E}\left[\left|X-\mu\right|^3\right]$. Then, the characteristic
function $\varphi_X$ of $X$ satisfies: $$\left|\log\varphi_X(t)-\left(i\mu
t-\frac{\sigma^2}{2}t^2\right)\right|\leq
\frac{\kappa|t|^3}{6}+\frac{\sigma^4t^4}{4}\quad\forall
t\in\left[-\frac{1}{\sigma},\frac{1}{\sigma}\right]\text{.}$$
\end{lem}

\begin{proof}
The result can be established by combining the elementary bound 
$$\left|\mathbb{E}\left[e^{i(X-\mu)t}-\sum_{k=0}^n\frac{(it)^n}{n!}(X-\mu)^n\right]\right|\leq \frac{|t|^{n+1}}{(n+1)!}\mathbb{E}\left[\left|X-\mu\right|^{n+1}\right]\quad\forall t\in\RR$$
and the fact that $z\in\mathbb{C}$, $|z|\leq 1/2$ implies
$\left|(\log(1+z)-z\right|\leq |z|^2$ (see \cite[p.~188]{Williams} for both).

\end{proof}

\begin{lem}\label{lemma:CharaN+bN^2}
Let $N=N(0,\sigma^2)$ and let $u,v\in\RR$. The random variable $uN+vN^2$ has a characteristic function that satisfies
$$\left|\log\varphi_{uN+vN^2}(t)-\left(iv\sigma^2t-\frac{u^2\sigma^2}{2}t^2\right)\right|\leq 2v^2\sigma^4t^2+2u^2|v|\sigma^4|t|^3\quad\forall t\in\left[-\frac{1}{4|v|\sigma^2},\frac{1}{4|v|\sigma^2}\right]\text{.}$$
\end{lem}

\begin{proof}
The CF $\varphi_{uN+vN^2}$ can be explicitly computed using standard complex analysis
$$\varphi_{uN+vN^2}(t)=\mathbb{E}\left[e^{i(uN+vN^2)t}\right]=\frac{1}{\sqrt{1-2iv\sigma^2t}}\exp\left(-\frac{u^2\sigma^2t^2}{2(1-2iv\sigma^2t)}\right)\quad \forall t\in\RR\text{.}$$
The rest can then be shown using the elementary inequalities: $z\in\mathbb{C}$,
$|z|\leq 1/2$ implies $\left|(\log(1+z)-z\right|\leq |z|^2$ and
$\left|1/(1-z)-1\right|\leq 2|z|$.  
\end{proof}

\bibliographystyle{alpha}
\bibliography{Bibliography}

\newcommand{\etalchar}[1]{$^{#1}$}
\begin{thebibliography}{DKPR87}

\bibitem[AC99]{AssarafCaffarel99}
R.~Assaraf and M.~Caffarel.
\newblock Zero-variance principle for monte carlo algorithms.
\newblock {\em Physical Review Letters}, 83(23):4682--4685, 1999.

\bibitem[Bar98]{Barthe98}
F.~Barthe.
\newblock Optimal young's inequality and its converse: a simple proof.
\newblock {\em Geometric {\&} Functional Analysis GAFA}, 8(2):234--242, 1998.

\bibitem[B{\'e}d07]{Bedard2007}
M.~B{\'e}dard.
\newblock Weak convergence of {M}etropolis algorithms for non-{I.I.D.} target
  distributions.
\newblock {\em Ann. App. Prob.}, 17(4):1222--1244, 2007.

\bibitem[BGJM11]{Handbook}
S.~Brooks, A.~Gelman, G.L. Jones, and X.-L. Meng, editors.
\newblock {\em Handbook of Markov chain Monte Carlo}.
\newblock Handbooks of Modern Stat. Methods. Chapman \& Hall/CRC, Boca Raton,
  FL, 2011.

\bibitem[BPR{\etalchar{+}}13]{HMCScaling}
A.~Beskos, N.~Pillai, G.~Roberts, J.-M. Sanz-Serna, and A.~Stuart.
\newblock Optimal tuning of the hybrid monte carlo algorithm.
\newblock {\em Bernoulli}, 19(5A):1501--1534, 2013.

\bibitem[BR08]{BedardRosenthal2008}
M.~B{\'e}dard and J.S. Rosenthal.
\newblock Optimal scaling of {M}etropolis algorithms: Heading toward general
  target distributions.
\newblock {\em Canadian Journal of Statistics}, 36(4):483--503, 2008.

\bibitem[BR17]{ZigZag}
J.~Bierkens and G.O. Roberts.
\newblock A piecewise deterministic scaling limit of lifted
  {M}etropolis--{H}astings in the {C}urie--{W}eiss model.
\newblock {\em Ann. App. Prob.}, 27(2):846--882, 2017.

\bibitem[BRS09]{Optimal_Scaling_Var}
A.~Beskos, G.O. Roberts, and A.~Stuart.
\newblock Optimal scalings for local {M}etropolis--{H}astings chains on
  nonproduct targets in high dimensions.
\newblock {\em Ann. App. Prob}, 19(3):863--898, 2009.

\bibitem[DK12]{petros}
P.~Dellaportas and I.~Kontoyiannis.
\newblock Control variates for estimation based on reversible markov chain
  monte carlo samplers.
\newblock {\em Journal of the Royal Statistical Society: Series B (Statistical
  Methodology)}, 74(1):133--161, 2012.

\bibitem[DKPR87]{HMC}
S.~Duane, A.D. Kennedy, B.J. Pendleton, and D.~Roweth.
\newblock Hybrid monte carlo.
\newblock {\em Physics Letters B}, 195(2):216--222, 1987.

\bibitem[DLP16]{Pavliotis}
A.B. Duncan, T.~Leli\`evre, and G.A. Pavliotis.
\newblock Variance reduction using nonreversible {L}angevin samplers.
\newblock {\em Journal of Statistical Physics}, 163:457, 2016.

\bibitem[DRVZ16]{fMALA}
A.~Durmus, G.O. Roberts, G.~Vilmart, and K.~Zygalakis.
\newblock Fast {L}angevin based algorithm for {MCMC} in high dimensions.
\newblock 2016.
\newblock To appear in \textit{Ann. App. Prob.}

\bibitem[EL03]{EichelsbacherLowe02}
P.~Eichelsbacher and M.~L\"owe.
\newblock Moderate deviations for i.i.d.\ random variables.
\newblock {\em ESAIM Probab. Stat.}, 7:209--218, 2003.

\bibitem[Gey92]{Geyer92}
C.J. Geyer.
\newblock Practical {M}arkov chain {M}onte {C}arlo.
\newblock {\em Statistical Science}, pages 473--483, 1992.

\bibitem[Hen97]{henderson}
S.G. Henderson.
\newblock {\em Variance Reduction Via an Approximating Markov Process}.
\newblock PhD thesis, Department of Operations Research, Stanford University,
  1997.

\bibitem[JH00]{jarner}
S.F. Jarner and E.~Hansen.
\newblock Geometric ergodicity of {M}etropolis algorithms.
\newblock {\em Stochastic Process. Appl.}, 85(2):341--361, 2000.

\bibitem[JLM15]{Jourdain_Scaling_RWM}
B.~Jourdain, T.~Leli\`evre, and B.~Miasojedow.
\newblock Optimal scaling for the transient phase of the random walk
  {M}etropolis algorithm: The mean-field limit.
\newblock {\em Ann. App. Prob.}, 25(4):2263--2300, 2015.

\bibitem[Kol06]{Kolassa}
J.E. Kolassa.
\newblock {\em Series approximation methods in statistics}, volume~88.
\newblock Springer Science \& Business Media, 2006.

\bibitem[KV86]{KipnisVaradhan86}
C.~Kipnis and S.~R.~S. Varadhan.
\newblock Central limit theorem for additive functionals of reversible {M}arkov
  processes and applications to simple exclusions.
\newblock {\em Comm. Math. Phys.}, 104(1):1--19, 1986.

\bibitem[Mey08]{Meyn_Control}
S.~Meyn.
\newblock {\em Control techniques for complex networks}.
\newblock Cambridge University Press, Cambridge, 2008.

\bibitem[MPS12]{Optimal_Scaling_RWM_hilbert}
J.C. Mattingly, N.S. Pillai, and A.~Stuart.
\newblock Diffusion limits of the {R}andom walk {M}etropolis algorithm in high
  dimensions.
\newblock {\em Ann. App. Prob}, 22(3):881--930, 2012.

\bibitem[MT09]{tweedie}
S.~Meyn and R.L. Tweedie.
\newblock {\em Markov chains and stochastic stability}.
\newblock Cambridge University Press, Cambridge, second edition, 2009.

\bibitem[MV16]{MV}
A.~Mijatovi\'c and J.~Vogrinc.
\newblock On the {P}oisson equation for {M}etropolis-{H}astings chains.
\newblock {\em arXiv preprint arXiv:1511.07464}, 2016.
\newblock To appear in \textit{Bernoulli}.

\bibitem[OGC17]{OatesGirolami2017}
C.J. Oates, M.~Girolami, and N.~Chopin.
\newblock Control functionals for monte carlo integration.
\newblock {\em Journal of the Royal Statistical Society: Series B (Statistical
  Methodology)}, 79(3):695--718, 2017.

\bibitem[PMG14]{MiraGirolami2014}
T.~Papamarkou, A.~Mira, and M.~Girolami.
\newblock Zero variance differential geometric markov chain monte carlo
  algorithms.
\newblock {\em Bayesian Analysis}, 9(1):97--128, 2014.

\bibitem[PST12]{Optimal_Scaling_MALA_hilbert}
N.S. Pillai, A.~Stuart, and A.H. Thi\'ery.
\newblock Optimal scalings and diffusion limits for the {L}angevin algorithm in
  high dimensions.
\newblock {\em Ann. App. Prob}, 22(6):2320--2356, 2012.

\bibitem[RGG97]{RobertsGelmanGilks97}
G.O. Roberts, A.~Gelman, and W.~R. Gilks.
\newblock Weak convergence and optimal scaling of random walk {M}etropolis
  algorithms.
\newblock {\em Ann. Appl. Probab.}, 7(1):110--120, 1997.

\bibitem[RR97]{RobertsRosenthal97}
G.O. Roberts and J.S. Rosenthal.
\newblock Geometric ergodicity and hybrid {M}arkov chains.
\newblock {\em Electron. Comm. Probab.}, 2:no.\ 2, 13--25 (electronic), 1997.

\bibitem[RR98]{MALAScaling}
G.O. Roberts and J.S. Rosenthal.
\newblock Optimal scaling of discrete approximations to {L}angevin diffusions.
\newblock {\em Journal of the Royal Statistical Society: Series B (Statistical
  Methodology)}, 60(1):255--268, 1998.

\bibitem[RR01]{RobertsRosenthal01}
G.O. Roberts and J.S. Rosenthal.
\newblock Optimal scaling for various {M}etropolis-{H}astings algorithms.
\newblock {\em Statistical science}, 16(4):351--367, 2001.

\bibitem[Tie94]{tierney}
L.~Tierney.
\newblock Markov chains for exploring posterior distributions.
\newblock {\em Ann. Statist.}, 22(4):1701--1762, 1994.
\newblock With discussion and a rejoinder by the author.

\bibitem[Wil91]{Williams}
D.~Williams.
\newblock {\em Probability with martingales}.
\newblock Cambridge university press, 1991.

\end{thebibliography}

\end{document}